\documentclass[11pt]{article}

\usepackage{latexsym}
\usepackage{amssymb}
\usepackage{amsthm}
\usepackage{amscd}
\usepackage{amsmath}
\usepackage{tikz}
\usepgflibrary{arrows}

\newtheorem{theorem}{Theorem}[section]

\newtheorem{lemma}{Lemma}[section]

\newtheorem{proposition}{Proposition}[section]
\newtheorem{corollary}{Corollary}[section]

\theoremstyle{definition}
\newtheorem*{example}{Example}
\newtheorem*{remark}{Remark}

\numberwithin{equation}{section}

\newcommand{\FF}{\mathbb{F}}    \newcommand{\CC}{\mathbb{C}}

 \def\ZZ{\mathbb{Z}} 
       \def\Res{\mathrm{Res}} \def\dd{\displaystyle}   
\def\vphi{\varphi}  
  \def\wt{\mathrm{wt}}
  
  \def\scscs{\scriptscriptstyle} \def\scs{\scriptstyle}

\newcommand{\cX}{\mathcal{X}}
\newcommand{\cS}{\mathcal{S}}

\newcommand{\cI}{\mathcal{I}}
\newcommand{\cC}{\mathcal{C}}
\newcommand{\cK}{\mathcal{K}}
\newcommand{\cM}{\mathcal{M}}

\def\cP{\mathcal{P}}

\newcommand{\cR}{\mathcal{R}}
\newcommand{\cA}{\mathcal{A}}
\newcommand{\One}{{1\hspace{-.14cm} 1}}
\newcommand{\one}{{1\hspace{-.11cm} 1}}

\newcommand{\larc}[1]{\hspace{-.4ex}\overset{#1}{\frown}\hspace{-.4ex}}
\newcommand{\slarc}[1]{\overset{#1}{\frown}}

\makeatletter
\renewcommand{\@makefnmark}{\mbox{\textsuperscript{}}}
\makeatother

\def\adots{\mathinner{\mkern2mu\raise0pt\hbox{.}  
\mkern2mu\raise4pt\hbox{.}\mkern1mu
\raise7pt\vbox{\kern7pt\hbox{.}}\mkern1mu}}

\allowdisplaybreaks[1]

\begin{document}

\title{Nonzero coefficients in restrictions and tensor\\  products of supercharacters of $U_n(q)$}
\author{Stephen Lewis\\ Department of Mathematics\\ University of Washington\\  \textsf{stedalew@u.washington.edu}
\and
Nathaniel Thiem\\ Department of Mathematics\\ University of Colorado at Boulder\\
 \textsf{thiemn@colorado.edu}}

\date{}

\maketitle

\begin{abstract} 
The standard supercharacter theory of the finite unipotent upper-triangular matrices $U_n(q)$ gives rise to a beautiful combinatorics based on set partitions.  As with the representation theory of the symmetric group, embeddings of $U_m(q)\subseteq U_n(q)$ for $m\leq n$ lead to branching rules.  Diaconis and Isaacs established that the restriction of a supercharacter of $U_n(q)$ is a nonnegative integer linear combination of supercharacters of $U_m(q)$ (in fact, it is polynomial in $q$).   In a first step towards understanding the combinatorics of coefficients in the branching rules of the supercharacters of $U_n(q)$, this paper characterizes when a given coefficient is nonzero in the restriction of a supercharacter and the tensor product of two supercharacters.  These conditions are given uniformly in terms of complete matchings in bipartite graphs.
\end{abstract}

\section{Introduction}

The representation theory of the finite groups of unipotent upper-triangular matrices $U_n(\FF_q)$ has traditionally been a hard problem, where even enumerating the irreducible representations is a well-known wild problem.   In fact, it is not even known if the number of irreducible representations is polynomial in $q$ (the Higman conjecture \cite{Hi60} suggests the affirmative).  However, Andr\'e \cite{An95,An99,An01,An02} and later Yan  \cite{Ya01} demonstrated if one decomposes the regular representation into ``nearly irreducible" pieces (called super-representations) rather than the usual irreducible pieces, one obtains a theory that is far more tractable with beautiful combinatorial underpinnings.  Later, work of Arias-Castro, Diaconis and Stanley \cite{ADS04} demonstrated that this theory could even be used in place of the usual representation theory in an application to random walks, and for more general supercharacter theories,   Otto \cite{Ot09} has shown that they can be used to bound nilpotence classes of nilpotent algebras.

While it has been a guiding principle that the supercharacter theory of $U_n(q)$ is analogous to the representation theory of the symmetric group $S_n$, many $U_n$-analogues of $S_n$ results remain to be worked out.  Some of the known observations include
\begin{enumerate}
\item[(a)] The irreducible characters of $S_n$  are indexed by partitions of $n$, and the supercharacters of $U_n(\FF_q)$ are indexed by (a $q$-analogue of) set partitions of $\{1,2,\ldots, n\}$ \cite{An95,Ya01,ADS04},
\item[(b)] For $S_n$, Young subgroups are a natural family of subgroups which give the corresponding character rings a Hopf structure through induction and restriction.  Similarly, \cite{Th08} defines an analogue to Young subgroups for $U_n(q)$, noting that while in the $S_n$-case the particular embedding of the subgroup typically does not matter, in the $U_n(q)$-case it is critical \cite{MT09,TV09}.  These subgroups are indexed by set-partitions instead of integer partitions.
\item[(c)] As an algebra, the ring of symmetric functions model restriction/induction branching rules for the characters of $S_n$ considered simultaneously for all $n\geq 0$.  The corresponding ring for the supercharacters of $U_n(\FF_q)$ seems to be the ring of symmetric functions in non-commuting variables \cite{Th08}.  
\end{enumerate}
This paper attempts to better understand the combinatorics of branching rule coefficients for the supercharacters of $U_n(\FF_q)$.  

In the symmetric group case, the irreducible character $\chi^\mu\times \chi^\nu$ appears in the decomposition of the restricted character $$\Res^{S_{|\lambda|}}_{S_{|\mu|}\times S_{|\nu|}}(\chi^\lambda)$$
only if $\mu, \nu \subset \lambda$.    For $U_n(q)$ this paper gives both necessary and sufficient conditions for analogous result, using the close relationship between tensor products and restrictions in this case.  In particular, the main results of this paper are
\begin{description}
\item[Theorem \ref{MainTheorem}.]  Given a set partition $\lambda$ and subgroup $U_K\subseteq U_n$, there is a  bipartite graph $\Gamma_K(\lambda)$ such that the trivial character appears in the decomposition of $\Res^{U_n}_{U_K}(\chi^\lambda)$ if and only if the graph has a complete matching.
\item[Theorem \ref{TensorResult}.] Given set partitions $\lambda$, $\mu$, and $\nu$, there is a  bipartite graph $\Gamma(\lambda,\mu,\nu)$ such that the $\chi^\nu$ appears in the decomposition of $\chi^\lambda\otimes\chi^\mu$ if and only if the graph has a complete matching.
\item[Theorem \ref{RestrictionResult}.] Given set partitions $\lambda$, $\mu$ and subgroup $U_K\subseteq U_n$, there is a  bipartite graph $\Gamma_K(\lambda,\mu)$ such that the $\chi^\mu$ appears in the decomposition of $\Res^{U_n}_{U_K}(\chi^\lambda)$ if and only if the graph has a complete matching.
\end{description}
The bipartite graphs referenced in all three results have a uniform construction as described in Section \ref{GeneralGraphSection}, and are remarkably easy to construct given the initial data.  However, the description of the bipartite graph in Theorem \ref{MainTheorem} is particularly nice, so we describe it separately in Section \ref{SectionMainResult}.  A fundamental part of extending Theorem \ref{MainTheorem} to Theorem \ref{TensorResult} and Theorem \ref{RestrictionResult} is a result that rewrites tensor products as restriction, as follows.
\begin{description}
\item[Theorem \ref{StraighteningTheorem}]  Given two supercharacters $\chi^\lambda$ and $\chi^\mu$ of $U_K$, there exists a supercharacter $\chi^\nu$ and groups $U_{K'}\subseteq U_L$ with $U_{K'}\cong U_K$ such that $\chi^\lambda\otimes \chi^\mu$ is the same character of $U_K$ as $q^{-r}\Res_{U_K'}^{U_L}(\chi^\nu)$ is of $U_{K'}$, where $r\in \ZZ_{\geq 0}$.
\end{description}

While these results do not give a complete description of the coefficients, \cite{DI08} showed that these coefficients are always positive integers (in fact, polynomial in $q$).  Thus, a more explicit understanding of the coefficients remains open.    We do, however, give explicit coefficients of the trivial character for a family of ``nice" examples in Theorem \ref{ExplicitCoefficientsTheorem}.  This result then also has generalizations for non-trivial coefficients in the same way that Theorems \ref{TensorResult} and \ref{RestrictionResult} use Theorem \ref{MainTheorem}.

The supercharacter theory studied in this paper is a particular example of a supercharacter theory that has a more general construction on algebra groups \cite{DI08}.   Generalizations of this approach have also been studied by Andr\'e and Neto for maximal unipotents subgroups in other Lie types \cite{AN06,AN09}.  This study of a particular supercharacter theory is somewhat different from recent work by \cite{He08}, which attempts to find all the possible supercharacter theories for a given finite group.  

\subsubsection*{Acknowledgements}

Part of this work (in particular Theorem \ref{MainTheorem}) was in Lewis' undergraduate honors thesis at the University of Colorado at Boulder.  Thiem was supported in part by an NSF grant DMS-0854893.

\section{Preliminaries}

This section sets up the necessary combinatorics of set partitions, which differs from some of the more standard formulations.  From this point of view, the parts of the set partition are less important than the relative sizes of the numbers in the same part.  We then review the definition of a supercharacter theory, and recall the specific supercharacter theory of interest for the finite groups of unipotent upper-triangular matrices, as developed by Andr\'e \cite{An95,An99,An01,An02} and Yan \cite{Ya01}.

\subsection{Set partition combinatorics}

Fix a prime power $q$, and let $\FF_q$ be the finite field with $q$ elements with additive group $\FF_q^+$ and multiplicative group $\FF_q^\times$.

For a  finite subset $K\subseteq \ZZ_{\geq 1}$, let
$$\cA_K(q)=\{i\larc{a}j \mid i,j\in K, i<j,a\in\FF_q^\times\},$$
and 
$$\cA(q)=\bigcup_{K\subseteq \ZZ_{\geq 1}\atop |K|<\infty} \cA_K(q),$$
where $\cA_\emptyset(q)=\{\emptyset\}$.  We will refer to the non-emptyset elements of $\cA(q)$ as \emph{arcs}.

Let
$$\cM(q)=\bigcup_{K\subseteq \ZZ_{\geq 1}\atop |K|<\infty} \cM_K(q),\qquad \text{where}\qquad \cM_K(q)=\{\text{finite multisets in $\cA_K(q)$}\}.$$
For $\lambda\in \cM_K(q)$ and $j,k\in K$, let
\begin{equation}\label{MultiSetStats}
R_{jk}=\{j\larc{a}k\in \lambda\},\qquad
m_{jk}(\lambda)=|R_{jk}|, \qquad \text{and}\qquad 
\wt_{jk}(\lambda)=\sum_{j\slarc{a}k\in R_{jk}} a \in \FF_q,
\end{equation}
where $R_{jk}$ is a multiset.
For example, if 
\begin{equation*}
\lambda=\{1\larc{a}4,1\larc{b}4,1\larc{c}4,3\larc{d}5\}=
\begin{tikzpicture}[baseline=.1cm]
	\foreach \x in {1,...,5} 
		\node (\x) at (\x,0) [inner sep=0pt] {$\bullet$};
	\foreach \x in {1,...,5} 
		\node at (\x,-.25) {$\scs\x$};
	\draw (1) .. controls (1.75,.5) and (3.25,.5) ..  node [auto] {$\scs a$} (4); 
	\draw (1) .. controls (1.75,1) and (3.25,1) .. node [auto] {$\scs b$} (4); 
	\draw (1) .. controls (1.75,1.5) and (3.25,1.5) .. node [auto] {$\scs c$} (4); 
	\draw (3) .. controls (3.5,.75) and (4.5,.75) .. node [auto] {$\scs d$}  (5);
\end{tikzpicture},\qquad \text{for $a,b,c,d\in \FF_q^\times$},
\end{equation*}
then
$$ R_{14}=\{1\larc{a}4,1\larc{b}4,1\larc{c}4\},\quad
m_{14}(\lambda)  =3,\quad
\wt_{14}(\lambda) = a+b+c,\quad \text{and}\quad
R_{23}= \emptyset.
$$
Note we will use a diagrammatic representation of multisets $\lambda\in \cM_K(q)$, by associating to each element of $K$ a node (usually arranged along a horizontal line), and each arc $i\larc{a}j$ in $\lambda$ becomes a labeled edge connecting node $i$ to node $j$.

A \emph{$q$-set partition} of $K$ is a multiset $\lambda\in \cM_K(q)$ such that if $i\larc{a}l,j\larc{b}k\in \lambda$ are two distinct arcs, then $i\neq j$ and $k\neq l$.  Let
$$\cS(q)=\bigcup_{K\subseteq\ZZ_{\geq 1}\atop |K|<\infty} \cS_K(q),\qquad \text{where}\qquad \cS_K(q)=\{\text{$q$-set partitions in $\cM_K(q)$}\}.$$
Note that $2$-set partitions of $K$ are set partitions $\lambda$ of $K$ by the rule that $i$ and $j$ are in the same part if there is a sequence of arcs $i\larc{1}j_1,j_1\larc{1}j_2,\ldots,j_{m-1}\larc{1}j\in \lambda$.  That is, the parts of the set partitions are the connected components of the diagrammatic representation of the $2$-set partition.  For example,
$$
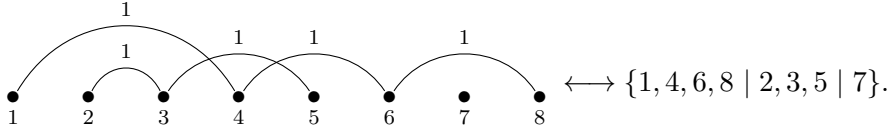
\begin{tikzpicture}[baseline=.1cm]
	\foreach \x in {1,...,8} 
		\node (\x) at (\x,0) [inner sep=0pt] {$\bullet$};
	\foreach \x in {1,...,8} 
		\node at (\x,-.25) {$\scs\x$};
	\draw (1) .. controls (1.75,1.25) and (3.25,1.25) ..  node [auto] {$\scs 1$} (4); 
	\draw (2) .. controls (2.25,.5) and (2.75,.5) .. node [auto] {$\scs 1$} (3); 
	\draw (4) .. controls (4.5,.75) and (5.5,.75) .. node [auto] {$\scs 1$} (6); 
	\draw (3) .. controls (3.5,.75) and (4.5,.75) .. node [auto] {$\scs 1$}  (5);
	\draw (6) .. controls (6.5,.75) and (7.5,.75) .. node [auto] {$\scs 1$}  (8);
\end{tikzpicture} \longleftrightarrow \{1,4,6,8\mid 2,3,5\mid 7\}.$$
In this sense, $q$-set partitions are a $q$-analogue of set partitions (although, strictly speaking, they are $(q-1)$-analogues of set partitions).

Let $\lambda\in \cM(q)$.  A \emph{conflict in $\lambda$ over $K$} is either
\begin{enumerate}
\item[(CL)] A pair of distinct arcs $i\larc{a}l,j\larc{b}k\in \lambda$ such that $i=j$ and $k<l$,
\item[(CR)] A pair of distinct arcs $i\larc{a}l,j\larc{b}k\in \lambda$ such that $i<j$ and $k=l$,
\item[(CB)] A pair of distinct arcs $i\larc{a}l,j\larc{b}k\in \lambda$ such that $i=j$ and $k=l$,
\item[(CN)] A  nonempty multiset $\{i\larc{a}k\in \lambda\mid i=j\text{ or } k=j\}$ for some $j\notin K$.
\end{enumerate}

\begin{example}
The multiset 
$$\lambda=
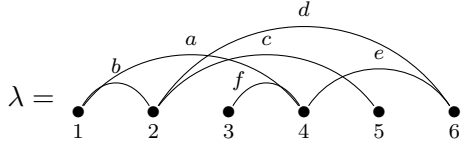
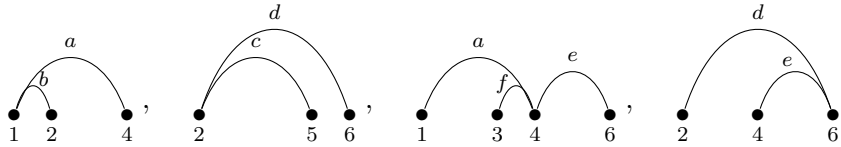
\begin{tikzpicture}[baseline=.1cm]
	\foreach \x in {1,...,6} 
		\node (\x) at (\x,0) [inner sep=0pt] {$\bullet$};
	\foreach \x in {1,...,6} 
		\node at (\x,-.25) {$\scs\x$};
	\draw (1) .. controls (1.75,1) and (3.25,1) ..  node [auto] {$\scs a$} (4); 
	\draw (1) .. controls (1.3,.5) and (1.7,.5) ..  node [auto] {$\scs b$} (2); 
	\draw (2) .. controls (2.75,1) and (4.25,1) .. node [auto] {$\scs c$} (5); 
	\draw (2) .. controls (3,1.5) and (5,1.5) .. node [auto] {$\scs d$} (6); 
	\draw (4) .. controls (4.5,.75) and (5.5,.75) .. node [auto] {$\scs e$}  (6);
	\draw (3) .. controls (3.3,.5) and (3.7,.5) ..  node [above=-1pt,pos=.1] {$\scs f$} (4);
\end{tikzpicture}
$$
has conflicts
$$
\begin{tikzpicture}[baseline=.1cm]
	\foreach \x in {1,2,4} 
		\node (\x) at (\x/2,0) [inner sep=0pt] {$\bullet$};
	\foreach \x in {1,2,4} 
		\node at (\x/2,-.25) {$\scs\x$};
	\draw (1) .. controls (1.75/2,1) and (3.25/2,1) ..  node [auto] {$\scs a$} (4); 
	\draw (1) .. controls (1.3/2,.5) and (1.7/2,.5) ..  node [above=-1pt,pos=.8] {$\scs b$} (2); 
\end{tikzpicture}
,\quad 
\begin{tikzpicture}[baseline=.1cm]
	\foreach \x in {2,5,6} 
		\node (\x) at (\x/2,0) [inner sep=0pt] {$\bullet$};
	\foreach \x in {2,5,6} 
		\node at (\x/2,-.25) {$\scs\x$};
	\draw (2) .. controls (2.75/2,1) and (4.25/2,1) .. node [auto] {$\scs c$} (5); 
	\draw (2) .. controls (3/2,1.5) and (5/2,1.5) .. node [auto] {$\scs d$} (6); 
\end{tikzpicture},
\quad 
\begin{tikzpicture}[baseline=.1cm]
	\foreach \x in {1,3,4,6} 
		\node (\x) at (\x/2,0) [inner sep=0pt] {$\bullet$};
	\foreach \x in {1,3,4,6} 
		\node at (\x/2,-.25) {$\scs\x$};
	\draw (1) .. controls (1.75/2,1) and (3.25/2,1) ..  node [auto] {$\scs a$} (4); 
	\draw (4) .. controls (4.5/2,.75) and (5.5/2,.75) .. node [auto] {$\scs e$}  (6);
	\draw (3) .. controls (3.3/2,.5) and (3.7/2,.5) ..  node [above=-1pt,pos=.1] {$\scs f$} (4);
\end{tikzpicture},
\quad
\begin{tikzpicture}[baseline=.1cm]
	\foreach \x in {2,4,6} 
		\node (\x) at (\x/2,0) [inner sep=0pt] {$\bullet$};
	\foreach \x in {2,4,6} 
		\node at (\x/2,-.25) {$\scs\x$};
	\draw (2) .. controls (3/2,1.5) and (5/2,1.5) .. node [auto] {$\scs d$} (6); 
	\draw (4) .. controls (4.5/2,.75) and (5.5/2,.75) .. node [above=-1pt,pos=.4] {$\scs e$}  (6);
\end{tikzpicture}
$$
over $\{ 1,2,3,5,6\}$, where the conflicts are of type (CL), (CL), (CN), and (CR). 
\end{example}

Conflicts are instances in a multiset which violate the conditions of membership in the set $\cS_K(q)$.  Theorem \ref{ArcRestriction} and Theorem \ref{ArcTensorProduct}, below, may be viewed as representation theoretic ``straightening rules" for resolving such conflicts in the context of restricting and taking tensor products.

\begin{remark}
There are a variety of $q$-analogues of set partition or Stirling numbers in the literature.  This particular $q$-analogue of set partitions is different from the one introduced \cite{HT09} and only seems to appear in connection with supercharacters.   There is also a standard construction for $q$-Stirling numbers (see for example \cite{GR86}), where the Stirling number $S(n,k;q)$ is defined by ``$q$-counting" the number of elements of $\cS_n(2)$ with $n-k$ arcs.  If we let $S_q(n,k)$ be the number of elements of $\cS_n(q)$ with $n-k$ elements, we obtain a recursion
$$\cS_q(n,k)=\cS_q(n-1,k-1)+k(q-1)\cS_q(n-1,k),$$
which is different from  the recursion for $S(n,k;q)$ in \cite{GR86}.
\end{remark}

\subsection{Supercharacters of $U_n(q)$}

Supercharacters were first studied by Andr\'e in relation to $U_n(q)$ as a way to find some more tractable way to understand the representation theory of $U_n(q)$.  Diaconis and Isaacs \cite{DI08} then generalized the concept to arbitrary finite groups, and we reproduce a version of this more general definition below.  

A \emph{supercharacter theory} of a finite group $G$ is a pair $(\cK,\cX)$ where $\cK$ is a partition of $G$ such that 
\begin{enumerate}
\item[(a)] Each $K\in\cK$ is a union of conjugacy classes,
\item[(b)] The identity element of $G$ is in its own part in $\cK$,
\end{enumerate} 
and $\cX$ is a set of characters of $G$ such that 
\begin{enumerate}
\item[(a)] For each irreducible character $\psi$ of $G$ there is a unique $\chi\in\cX$ such that $\langle\chi,\psi\rangle\neq 0$,
\item[(b)] The trivial character $\One\in\cX$,
\item[(c)] The characters of $\cX$ are constant on the parts of $\cK$,
\item[(d)] $|\cK|=|\cX|$.
\end{enumerate}
We will refer to the parts of $\cK$ as \emph{superclasses} and the characters of $\cX$ as \emph{supercharacters}.  For more information on the implications of these axioms see \cite{DI08} (including some redundancies in the definition).

For $n\in\ZZ_{\geq 1}$, let $M_n(\FF_q)$ be the ring of $n\times n$ matrices with entries in the finite field $\FF_q$ with $q$ elements.  Let
$$U_n(q)=\{u\in M_n(\FF_q)\mid u_{ji}=0, u_{ii}=1, u_{ij}\in \FF_q, i<j\}$$
be the group of unipotent upper-triangular matrices.  For $K\subseteq \{1,2,\ldots, n\}$, let
$$U_K(q)=\{u\in U_n(q)\mid u_{ij}=0, i<j, \text{unless $i,j\in K$}\}.$$
Note that $U_K(q)\cong U_{|K|}(q)$.

For $U_K(q)$ there is a standard example of a supercharacter theory developed by Andr\'e and Yan, where $\cK$ and $\cX$ are indexed by $q$-set partitions of $K$.  For the purpose of this paper it suffices to recall the definition of the supercharacters.  Fix a nontrivial group homomorphism, 
$$\vartheta:\FF_q^+\longrightarrow \CC^\times.$$
For $\lambda\in \cS_K(q)$, there is a supercharacter $\chi^\lambda$ given by
\begin{equation} \label{ArcDecomposition}
\chi^\lambda=\bigotimes_{i\slarc{a}l\in \lambda} \chi^{i\slarc{a}l},
\end{equation}
where each $\chi^{i\slarc{a}l}$ is an irreducible character of $U_K(q)$ whose value on the superclass indexed by $\mu\in \cS_K(q)$ is 
$$\chi^{i\slarc{a}l}(\mu)=\left\{\begin{array}{ll} 0, & \text{if $j\larc{b}k\in\mu$ with $i=j<k<l$ or $i<j<k=l$,}\\
\dd \frac{q^{|\{i<j<l\mid j\in K\}|}}{q^{|\{i<j<k<l\mid j\slarc{b}k\in \mu\}|}} \vartheta(a \wt_{il}(\mu)), &  \text{otherwise.}\\ 
\end{array}\right.$$
It can be quickly verified that the linear supercharacters of $U_K(q)$ correspond to $\lambda\in\cS_K(q)$ with $i\larc{a}l\in\lambda$ implies $\{i<j<l\mid j\in K\}=\emptyset$;  the trivial character is $\chi^\emptyset$.  

The superclass $\{1\}$ is indexed by $\emptyset\in \cS_L(q)$.  Thus, for $\lambda\in \cM_L(q)$ with $L\subseteq  \ZZ_{\geq 1}$, the degree of $\chi^\lambda$ is 
$$\chi^\lambda(1)=\prod_{i\slarc{a}l\in \lambda} q^{|\{i<j<l\mid j\in L\}|}.$$ 
If $K\subseteq L$, then define 
\begin{equation}\label{DegreeRatio}
r_K^L(\lambda)=|\{(j,i\larc{a}l)\in L\times \lambda\mid i<j<l, j\notin K\}|.
\end{equation}
Note that if $\lambda\in \cM_L(q)\cap \cM_K(q)$ then $q^{r_K^L(\lambda)}$ is the ratio of the degrees of $\chi^\lambda$ as a character of $U_L(q)$ and $\chi^\lambda$ as a character of $U_K(q)$.  It therefore is a constant that frequently comes up in the restriction of supercharacters.  In fact, by inspection, if  $\lambda\in \cM_L(q)\cap \cM_K(q)$, then
\begin{equation}\label{AlreadyDownArcsRestriction}
\Res_{U_K(q)}^{U_L(q)}(\chi^\lambda)=q^{r_K^L(q)}\chi^\lambda.
\end{equation}

In general, supercharacters are orthogonal with respect to the usual inner product on class functions, and for $\lambda,\mu\in\cS_K(q)$,
\begin{equation}\label{OrthogonalityRelation}
\langle \chi^\lambda,\chi^\mu\rangle =\delta_{\lambda\mu}q^{|\cC(\lambda)|},\qquad \text{where} \quad \cC(\lambda)=\{(i\larc{a}k,j\larc{b}l)\in \lambda\times\lambda\mid i<j<k<l\},
\end{equation}
is the set of \emph{crossings} in $\lambda$.

The following two theorems describe local rules for computing restrictions and tensor products.  To use these rules for arbitrary set partitions, one uses the observations (\ref{ArcDecomposition}) and
$$\Res^{U_L}_{U_K}(\chi^\lambda)=\bigotimes_{i\slarc{a}l} \Res^{U_L}_{U_K}(\chi^{i\slarc{a}l}).$$
In principle, therefore, one can easily compute restrictions and tensor products in a recursive, algorithmic fashion  (see \cite{Th08} for a detailed description of this algorithm, and \cite{Le09} for an implementation of this algorithm in Python).  However, this algorithm does not give an obvious combinatorial interpretation of the resulting coefficients.

The first theorem describes how to restrict an arc from $U_L$ to a subgroup $U_K$, or how to resolve a (CN) conflict.

\begin{theorem}[\cite{Th08,TV09}]\label{ArcRestriction}
For $i\larc{a}l\in\cA_L(q)$ and $K\subseteq L$,
$$\Res_{U_K}^{U_L}(\chi^{i\slarc{a}l})=\left\{\begin{array}{ll} 
\dd  q^{r_K^L(i\slarc{a}l)}\chi^{i\slarc{a}l}, & \text{if $i,l\in K$,}\\
\dd q^{r_K^L(i\slarc{a}l)}\bigg(\chi^{\emptyset}+\sum_{i<j<l\atop j\in K,b\in\FF_q^\times} \chi^{j\slarc{b}l}\bigg), & \text{if $i\notin K, l\in K$,}\\
\dd q^{r_K^L(i\slarc{a}l)}\bigg(\chi^{\emptyset} + \sum_{i<k<l\atop k\in K,b\in\FF_q^\times} \chi^{i\slarc{b}k}\bigg), & \text{if $i\in K, l\notin K$,}\\
\dd q^{r_K^L(i\slarc{a}l)}\bigg((|K\cap [i,l]|(q-1)+1)\chi^\emptyset+(q-1)\hspace{-.35cm}\sum_{i<j<k<l\atop j,k\in K,b\in\FF_q^\times} \hspace{-.25cm} \chi^{j\slarc{b}k}\bigg), & \text{if $i,l\notin K$.}\end{array}\right.$$
\end{theorem}

The next theorem explains how to resolve conflicts of types (CL), (CR) or (CB).  It was originally proved in \cite{Ya01}, but a proof can also be found in \cite{Th08}.

\begin{theorem}[\cite{Ya01}]\label{ArcTensorProduct}
For $i\larc{a}l,j\larc{b}k\in\cA_K(q)$,
$$\chi^{\{i\slarc{a}l,j\slarc{b}k\}}=\left\{\begin{array}{ll} 
\dd \chi^{\{i\slarc{a}l,j\slarc{b}k\}}, & \text{if no conflict,}\\
\dd \chi^{i\slarc{a}l}+\sum_{i<j'<l\atop j'\in K,c\in\FF_q^\times} \chi^{\{i\slarc{a}l,j'\slarc{c}k\}}, & \text{if (CL), $k<l$,}\\
\dd \chi^{i\slarc{a}l} + \sum_{i<k'<l\atop k'\in K,c\in\FF_q^\times} \chi^{\{i\slarc{a}l,j\slarc{c}k'\}}, & \text{if (CR), $i<j$,}\\
\dd \chi^{\emptyset}+\sum_{i<k'<l\atop k'\in K,c\in\FF_q^\times} (\chi^{i\slarc{c}k'}+\chi^{k'\slarc{c}l})+\sum_{i<j',k'<l\atop j',k'\in K,c,d\in\FF_q^\times} \hspace{-.25cm} \chi^{\{i\slarc{c}j',k'\slarc{d}l\}}, & \text{if (CB), $a+b= 0$,}\\
\dd (|K\cap [i,l]|(q-1)+1)\chi^{i\slarc{a+b}l}+(q-1)\hspace{-.35cm}\sum_{i<j'<k'<l\atop j',k'\in K,c\in\FF_q^\times} \hspace{-.25cm} \chi^{\{i\slarc{a+b}l,j'\slarc{c}k'\}}, & \text{if (CB), $a+b\neq 0$.}\end{array}\right.$$
\end{theorem}

It follows from Theorem \ref{ArcTensorProduct} that for any $\lambda\in \cM_K(q)$, the product
$$\chi^\lambda=\bigotimes_{i\slarc{a}l\in \lambda} \chi^{i\slarc{a}l},$$
is a character of $U_K(q)$ and a $\ZZ_{\geq 0}$-linear combination of supercharacters.  The following proposition summarizes some of the observations of this section.

\begin{proposition}
Let $\lambda\in \cM_K(q)$.  Then
\begin{enumerate}
\item[(a)]  If $\lambda\in\cS_K(q)$, then $\chi^\lambda$ is a supercharacter of $U_K(q)$,
\item[(b)]  If  $\lambda\in\cS_K(q)$ and $\lambda$ has no crossings, then $\chi^\lambda$ is an irreducible character of $U_K(q)$.
\end{enumerate}
\end{proposition}

\section{Coefficient of trivial character}

This section investigates the coefficient of the trivial character $\One=\chi^\emptyset$ in the restriction from $U_L(q)$ to a subgroup $U_K(q)$.  In particular, Theorem \ref{MainTheorem} characterizes when the coefficient of $\One$ is nonzero in the restriction of a supercharacter.  Although the theorem seems somewhat specific, in later sections we will use it to analyze the coefficients of arbitrary supercharacters in both restrictions and tensor products.  

We begin with the statement of the main result.  To prove the theorem, we define a poset on $\cM(q)$ where two multisets are related if and only if there is a way to resolve conflicts using some combination of Theorems \ref{ArcRestriction} and \ref{ArcTensorProduct}.  We then use this poset to show that many arcs do not affect whether the coefficient is nonzero, and so may be removed.  Once we have eliminated the excess arcs, we obtain our main result.

\subsection{Main result} \label{SectionMainResult}

Given a set partition $\lambda\in \cS_L(q)$ and a subset $K\subseteq L$, let $\Gamma_K(\lambda)$ be the bipartite graph given by vertices 
\begin{align*}
V_\bullet &=\{i\larc{}j\in \lambda\mid i,j\in K\}\\
V_\circ &=\{i\larc{}j\in \lambda\mid i,j\notin K\},
\end{align*}
and an edge from $j\larc{}k\in V_\bullet$ to $i\larc{}l\in V_\circ$ if $i<j<k<l$.   Note that this graph has in general far fewer vertices than $\lambda$ has arcs.  The following theorem is the main result of the paper, and is a model for the remaining results in this paper.

\begin{theorem}\label{MainTheorem}
Let $K\subseteq L\subseteq \ZZ_{\geq 1}$ be finite sets, and let $\lambda\in \cS_L(q)$ be a $q$-set partition.   Then
$$\langle \Res_{U_K}^{U_L}(\chi^\lambda),\One\rangle\neq 0,$$
if and only if 
the graph $\Gamma_K(\lambda)$ has a complete matching from $V_\bullet$ to $V_\circ$.
\end{theorem}

\begin{remark}
The complete matching referred to in Theorem \ref{MainTheorem} is a one-sided matching.  That is, every element in $V_\bullet$ must be matched to a corresponding element in $V_\circ$, but there could potentially be elements of $V_\circ$ not matched to elements of $V_\bullet$.  For example, if 
$$\Gamma_1=
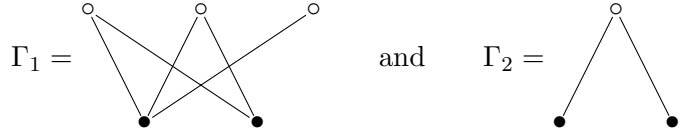
\begin{tikzpicture}[baseline=.75cm]
	\node (01) at (0,1.5) [inner sep=0pt] {$\circ$};
	\node (11) at (1.5,1.5) [inner sep=0pt] {$\circ$};
	\node (21) at (3,1.5) [inner sep=0pt] {$\circ$};
	\node (00) at (.75,0) [inner sep=0pt] {$\bullet$};
	\node (10) at (2.25,0) [inner sep=0pt] {$\bullet$};
	\draw (10) -- (01) -- (00) -- (21);
	\draw (00) -- (11) -- (10);
\end{tikzpicture}
\qquad\text{and}\qquad
\Gamma_2=
\begin{tikzpicture}[baseline=.75cm]
	\node (11) at (1.5,1.5) [inner sep=0pt] {$\circ$};
	\node (00) at (.75,0) [inner sep=0pt] {$\bullet$};
	\node (10) at (2.25,0) [inner sep=0pt] {$\bullet$};
	\draw (00) -- (11) -- (10);
\end{tikzpicture}$$
then $\Gamma_1$ has a complete matching from $V_\bullet$ to $V_\circ$ and $\Gamma_2$ does not.
\end{remark}

\begin{example}
If $K=\{1,4,5,6,7,9\}$ and 
$$\lambda=
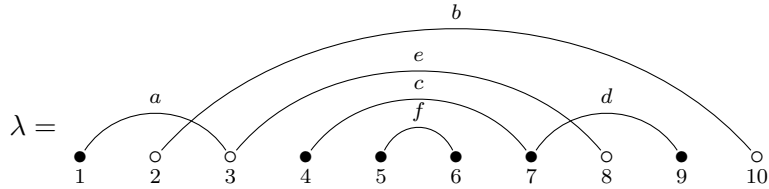
\begin{tikzpicture}[baseline=.3cm]
	\foreach \x in {1,4,5,6,7,9} 
		\node (\x) at (\x,0) [inner sep=0pt] {$\bullet$};
	\foreach \x in {2,3,8,10} 
		\node (\x) at (\x,0) [inner sep=0pt] {$\circ$};
	\foreach \x in {1,...,10} 
		\node at (\x,-.25) {$\scs\x$};
	\draw (1) .. controls (1.5,.75) and (2.5,.75) .. node [auto] {$\scs a$} (3); 
	\draw (2) .. controls (4,2.25) and (8,2.25) .. node [auto] {$\scs b$} (10); 
	\draw (4) .. controls (4.75,1) and (6.25,1) .. node [auto] {$\scs c$} (7);
	\draw (7) .. controls (7.5,.75) and (8.5,.75) ..  node [auto] {$\scs d$} (9);
	\draw (3) .. controls (4.25,1.5) and (6.75,1.5) .. node [auto] {$\scs e$} (8);
	\draw (5) .. controls (5.25,.5) and (5.75,.5) .. node [above=-1.5pt] {$\scs f$} (6);
\end{tikzpicture}$$
then $V_\bullet=\{4\larc{c}7,5\larc{f}6,7\larc{d}9\}$, $V_\circ=\{2\larc{b}10,3\larc{e}8\}$, and 
$$\Gamma_K(\lambda)=
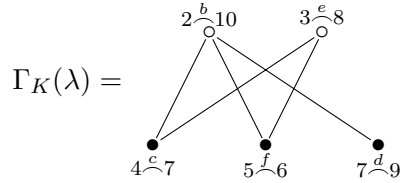
\begin{tikzpicture}[baseline=.75cm]
	\node (01) at (.75,1.5) [inner sep=0pt] {$\circ$};
	\node at (.75,1.75) [inner sep=0pt] {$\scs 2\slarc{b}10$};
	\node (11) at (2.25,1.5) [inner sep=0pt] {$\circ$};
	\node  at (2.25,1.75) [inner sep=0pt] {$\scs 3\slarc{e}8$};
	\node (00) at (0,0) [inner sep=0pt] {$\bullet$};
	\node at (0,-.25) [inner sep=0pt] {$\scs 4\slarc{c}7$};
	\node (10) at (1.5,0) [inner sep=0pt] {$\bullet$};
	\node at (1.5,-.25) [inner sep=0pt] {$\scs 5\slarc{f}6$};
	\node (20) at (3,0) [inner sep=0pt] {$\bullet$};
	\node  at (3,-.25) [inner sep=0pt] {$\scs 7\slarc{d}9$};
	\draw (11) -- (00) -- (01) -- (20);
	\draw (01) -- (10) -- (11);
\end{tikzpicture}
$$
Since this graph has no complete matchings from $V_\bullet$ to $V_\circ$, by Theorem \ref{MainTheorem} $\langle \Res_{U_K}^{U_L}(\chi^\lambda),\One\rangle=0$.
\end{example}

\subsection{The conflict poset}

There is a reasonably natural poset that helps model branching rules given by Theorem \ref{ArcRestriction} or Theorem \ref{ArcTensorProduct}.  As observed in Proposition \ref{PosetToCoefficient}, below,  two multi-partitions are comparable in the following poset if one appears in the decomposition of the other in some restriction or tensor product of characters.

 We say $\mu\prec \lambda$ if there exists a conflict $\gamma\subseteq \lambda$ over $K$ such that $\lambda-\gamma\subseteq \mu$ and  
for $\nu= \mu-(\lambda-\gamma)$, either 
 \begin{enumerate}
\item[(PL)] $\gamma=\{i\larc{a}l,i\larc{b}k\}$ with $k<l$ and $\nu\in\Big\{\{i\larc{a}l\},\{i\larc{a}l,j\larc{c}k\}\mid i<j<k,c\in\FF_q^\times\Big\}$,
\item[(PR)] $\gamma=\{i\larc{a}l,j\larc{b}l\}$ with $i<j$ and $\nu\in\Big\{\{i\larc{a}l\},\{i\larc{a}l,j\larc{c}k\}\mid j<k<l,c\in\FF_q^\times\Big\}$,
\item[(PB)] $\gamma=\{i\larc{a}l,i\larc{b}l\}$ with 
\begin{itemize}
\item $a+b= 0$ and $\nu\in\Big\{\emptyset,\{i\larc{c}k\},\{j\larc{c}l\}, \{i\larc{c}k,j\larc{d}l\}\mid i<j<l,i<k<l,c,d\in\FF_q^\times\Big\}$,
\item $a+b\neq 0$ and $\nu\in\Big\{\{i\larc{a+b}l\},\{i\larc{a+b}l,j\larc{c}k\}\mid i<j<k<l,c\in\FF_q^\times\Big\}$.
\end{itemize}
\item[(PN)] There exists $j\notin K$ such that $\gamma=\{i\larc{a}k\in \lambda\mid i=j\text{ or } k=j\}$, and there exists an injective function $\iota:\nu\rightarrow \gamma$  such that if $\iota(k\larc{a}l)=j\larc{b}l'$ then $l=l'$ and $j<k$; and if  $\iota(h\larc{a}i)=h'\larc{b}j$, then $h=h'$ and $i<j$.
\end{enumerate}
Extending this relation transitively gives a partial order $\cP$ for all of $\cM(q)$, which we will refer to as the \emph{conflict poset}.   Roughly speaking, as one moves down the poset one resolves conflicts by shrinking arcs or making them disappear altogether (ie. using Theorems \ref{ArcRestriction} and \ref{ArcTensorProduct}), so the rules (PL), (PR), (PB), and (PN) match up with the corresponding conflicts (CL), (CR), (CB) and (CN).

\begin{remark}
Note that even though (PN) is the only condition that explicitly spells it out, if $\mu\preceq \lambda$ then there is an injective function $\iota:\mu\rightarrow \lambda$ that satisfies $\iota(j\larc{b}k)=i\larc{a}l$ implies $i\leq j<k\leq l$.  We will use this to keep track of arcs as we follow a string of relations in $\cP$.  
\end{remark}

\begin{example}  If
$$\lambda=
\begin{tikzpicture}[baseline=.1cm]
	\foreach \x in {1,...,6} 
		\node (\x) at (\x,0) [inner sep=0pt] {$\bullet$};
	\foreach \x in {1,...,6} 
		\node at (\x,-.25) {$\scs\x$};
	\draw (1) .. controls (1.75,1) and (3.25,1) ..  node [auto] {$\scs a$} (4); 
	\draw (2) .. controls (2.75,1) and (4.25,1) .. node [auto] {$\scs b$} (5); 
	\draw (4) .. controls (4.5,.75) and (5.5,.75) .. node [auto] {$\scs c$}  (6);
\end{tikzpicture}\in \cS_6(q)
\qquad\text{and}\qquad 
\mu=
\begin{tikzpicture}[baseline=.1cm]
	\foreach \x in {3,4,5} 
		\node (\x) at (\x,0) [inner sep=0pt] {$\bullet$};
	\foreach \x in {3,4,5} 
		\node at (\x,-.25) {$\scs\x$};
	\draw (3) .. controls (3.25,.5) and (3.75,.5) ..  node [auto] {$\scs d$} (4); 
\end{tikzpicture}\in \cS_{\{3,4,5\}}(q),
$$
then the interval between $\lambda$ and $\mu$ is
$$\begin{tikzpicture}
		\node (1B1) at (7.75,9.5) [inner sep=0pt] {};
		\foreach \x in {1,...,6} 
			\node (1\x) at (6+\x/2,10) [inner sep=0pt] {$\bullet$};
		\foreach \x in {1,...,6} 
			\node at (6+\x/2,9.75) {$\scs\x$};
		\draw (11) .. controls (6+1.75/2,10+1/2) and (6+3.25/2,10+1/2) ..  node [above] {$\scs a$} (14); 
		\draw (12) .. controls (6+2.75/2,10+1/2) and (6+4.25/2,10+1/2) .. node [above] {$\scs b$} (15); 
		\draw (14) .. controls (6+4.5/2,10+.75/2) and (6+5.5/2,10+.75/2) .. node [above] {$\scs c$}  (16);
		\node(15B1) at (5.25,7.5) [inner sep=0pt] {};
		\foreach \x in {2,...,6} 
			\node (2\x) at (3.5+\x/2,8) [inner sep=0pt] {$\bullet$};
		\foreach \x in {2,...,6} 
			\node at (3.5+\x/2,7.75) {$\scs\x$};
		\draw (22) .. controls (3.5+2.5/2,8+.75/2) and (3.5+3.5/2,8+.75/2) ..  node [above=-1pt] {$\scs a'$} (24); 
		\draw (22) .. controls (3.5+2.75/2,8+1.75/2) and (3.5+4.25/2,8+1.75/2) .. node (15T1) [above] {$\scs b$} (25); 
		\draw (24) .. controls (3.5+4.5/2,8+.75/2) and (3.5+5.5/2,8+.75/2) .. node [above] {$\scs c$}  (26);				
		\node (2B1) at (3.75,5.5) [inner sep=0pt] {};
		\foreach \x in {2,...,6} 
			\node (2\x) at (2+\x/2,6) [inner sep=0pt] {$\bullet$};
		\foreach \x in {2,...,6} 
			\node at (2+\x/2,5.75) {$\scs\x$};
		\draw (23) .. controls (2+3.25/2,6+.5/2) and (2+3.75/2,6+.5/2) ..  node [above] {$\scs d$} (24); 
		\draw (22) .. controls (2+2.75/2,6+1.5/2) and (2+4.25/2,6+1.5/2) .. node (2T1) [above] {$\scs b$} (25); 
		\draw (24) .. controls (2+4.5/2,6+.75/2) and (2+5.5/2,6+.75/2) .. node [above] {$\scs c$}  (26);
		\node (2B2) at (7,5.5) [inner sep=0pt] {};
		\foreach \x in {1,...,5} 
			\node (3\x) at (5.25+\x/2,6) [inner sep=0pt] {$\bullet$};
		\foreach \x in {1,...,5} 
			\node at (5.25+\x/2,5.75) {$\scs\x$};
		\draw (31) .. controls (5.25+1.75/2,6+1/2) and (5.25+3.25/2,6+1/2) ..  node [above] {$\scs a$} (34); 
		\draw (32) .. controls (5.25+2.75/2,6+1.5/2) and (5.25+4.25/2,6+1.5/2) .. node (2T2) [above] {$\scs b$} (35); 
		\draw (34) .. controls (5.25+4.25/2,6+.75/2) and (5.25+4.75/2,6+.75/2) .. node [above=-2pt,pos=.1] {$\scs c'$}  (35);
		\node (3B1) at (11,3.5) [inner sep=0pt] {};
		\foreach \x in {1,...,5} 
			\node (4\x) at (9.25+\x/2,4) [inner sep=0pt] {$\bullet$};
		\foreach \x in {1,...,5} 
			\node at (9.25+\x/2,3.75) {$\scs\x$};
		\draw (41) .. controls (9.25+1.75/2,4+1/2) and (9.25+3.25/2,4+1/2) ..  node [above] {$\scs a$} (44); 
		\draw (42) .. controls (9.25+2.75/2,4+1/2) and (9.25+4.25/2,4+1/2) .. node (3T1) [above] {$\scs b$} (45); 
		\node (2B3) at (9.75,5.5) [inner sep=0pt] {};
		\node (2B3S) at (11.25,5.75) [inner sep=0pt] {};
		\foreach \x in {1,3,4,5,6} 
			\node (5\x) at (8+\x/2,6) [inner sep=0pt] {$\bullet$};
		\foreach \x in {1,3,4,5,6} 
			\node at (8+\x/2,5.75) {$\scs\x$};
		\draw (51) .. controls (8+1.75/2,6+1/2) and (8+3.25/2,6+1/2) ..  node [above] {$\scs a$} (54); 
		\draw (54) .. controls (8+4.5/2,6+1.25/2) and (8+5.5/2,6+1.25/2) .. node [above] {$\scs c$} (56); 
		\draw (54) .. controls (8+4.25/2,6+.75/2) and (8+4.75/2,6+.75/2) .. node (2T3) [above] {$\scs b'$}  (55);
		\node (3B2) at (13.75,3.5) [inner sep=0pt] {};
		\foreach \x in {1,3,4,5,6} 
			\node (6\x) at (12+\x/2,4) [inner sep=0pt] {$\bullet$};
		\foreach \x in {1,3,4,5,6} 
			\node at (12+\x/2,3.75) {$\scs\x$};
		\draw (61) .. controls (12+1.75/2,4+1/2) and (12+3.25/2,4+1/2) ..  node (3T2) [above] {$\scs a$} (64); 
		\draw (64) .. controls (12+4.5/2,4+.75/2) and (12+5.5/2,4+.75/2) .. node [above] {$\scs c$} (66); 		
		\node (4B1) at (1.75,1.5) [inner sep=0pt] {};
		\foreach \x in {3,...,6} 
			\node (7\x) at (\x/2,2) [inner sep=0pt] {$\bullet$};
		\foreach \x in {3,...,6} 
			\node at (\x/2,1.75) {$\scs\x$};
		\draw (73) .. controls (3.25/2,2+.5/2) and (3.75/2,2+.5/2) ..  node [above] {$\scs d$} (74); 
		\draw (74) .. controls (4.25/2,2+.5/2) and (4.75/2,2+.5/2) .. node (4T1) [above] {$\scs b'$} (75); 
		\draw (74) .. controls (4.5/2,2+.75/2) and (5.5/2,2+.75/2) .. node [above] {$\scs c$}  (76);
		\node (5B1) at (9,-.5) [inner sep=0pt] {};
		\foreach \x in {3,...,6} 
			\node (8\x) at (7.25+\x/2,0) [inner sep=0pt] {$\bullet$};
		\foreach \x in {3,...,6} 
			\node at (7.25+\x/2,-.25) {$\scs\x$};
		\draw (83) .. controls (7.25+3.25/2,.5/2) and (7.25+3.75/2,.5/2) ..  node [above] {$\scs d$} (84); 
		\draw (84) .. controls (7.25+4.5/2,.75/2) and (7.25+5.5/2,.75/2) .. node (5T1) [above] {$\scs c$}  (86);
		\node (5B2) at (11.75,-.5) [inner sep=0pt] {};
		\foreach \x in {2,...,5} 
			\node (9\x) at (10+\x/2,0) [inner sep=0pt] {$\bullet$};
		\foreach \x in {2,...,5} 
			\node at (10+\x/2,-.25) {$\scs\x$};
		\draw (92) .. controls (10+2.75/2,1.5/2) and (10+4.25/2,1.5/2) ..  node (5T2) [above] {$\scs b$} (95); 
		\draw (93) .. controls (10+3.25/2,.5/2) and (10+3.75/2,.5/2) ..  node [above] {$\scs d$} (94); 
		\node (4B2) at (4.5,1.5) [inner sep=0pt] {};
		\foreach \x in {2,...,5} 
			\node (10\x) at (2.75+\x/2,2) [inner sep=0pt] {$\bullet$};
		\foreach \x in {2,...,5} 
			\node at (2.75+\x/2,1.75) {$\scs\x$};
		\draw (102) .. controls (2.75+2.75/2,2+2/2) and (2.75+4.25/2,2+2/2) ..  node (4T2) [above] {$\scs b$} (105); 
		\draw (103) .. controls (2.75+3.25/2,2+.5/2) and (2.75+3.75/2,2+.5/2) ..  node [above] {$\scs d$} (104); 
		\draw (104) .. controls (2.75+4.25/2,2+.5/2) and (2.75+4.75/2,2+.5/2) ..  node [above=-1pt,pos=.2] {$\scs c'$} (105); 
		\node (4B3) at (7.25,1.5) [inner sep=0pt] {};
		\foreach \x in {1,3,4,5} 
			\node (11\x) at (5.5+\x/2,2) [inner sep=0pt] {$\bullet$};
		\foreach \x in {1,3,4,5} 
			\node at (5.5+\x/2,1.75) {$\scs\x$};
		\draw (111) .. controls (5.5+1.75/2,2+1.5/2) and (5.5+3.25/2,2+1.5/2) ..  node (4T3) [above] {$\scs a$} (114); 
		\draw (114) .. controls (5.5+4.25/2,2+1.5/2) and (5.5+4.75/2,2+1.5/2) ..  node [above] {$\scs -c'$} (115); 
		\draw (114) .. controls (5.5+4.25/2,2+.5/2) and (5.5+4.75/2,2+.5/2) ..  node [above=-1pt] {$\scs c'$} (115); 
		\node (5B3) at (14.5,-.5) [inner sep=0pt] {};
		\foreach \x in {1,3,4,5} 
			\node (12\x) at (12.75+\x/2,0) [inner sep=0pt] {$\bullet$};
		\foreach \x in {1,3,4,5} 
			\node at (12.75+\x/2,-.25) {$\scs\x$};
		\draw (121) .. controls (12.75+1.75/2,1.5/2) and (12.75+3.25/2,1.5/2) ..  node (5T3) [above] {$\scs a$} (124); 
		\node (6B1) at (3.75,-2.5) [inner sep=0pt] {};
		\foreach \x in {3,4,5} 
			\node (13\x) at (2+\x/2,-2) [inner sep=0pt] {$\bullet$};
		\foreach \x in {3,4,5} 
			\node at (2+\x/2,-2.25) {$\scs\x$};
			\draw (133) .. controls (2+3.25/2,-2+.5/2) and (2+3.75/2,-2+.5/2) ..  node [above] {$\scs d$} (134); 
			\draw (134) .. controls (2+4.25/2,-2+1.5/2) and (2+4.75/2,-2+1.5/2) ..  node (6T1) [above] {$\scs -c'$} (135); 
			\draw (134) .. controls (2+4.25/2,-2+.5/2) and (2+4.75/2,-2+.5/2) ..  node [above=-1pt] {$\scs c'$} (135); 
		\foreach \x in {3,4,5} 
			\node (14\x) at (6+\x/2,-3.5) [inner sep=0pt] {$\bullet$};
		\foreach \x in {3,4,5} 
			\node at (6+\x/2,-3.75) {$\scs\x$};
			\draw (143) .. controls (6+3.25/2,-3.5+.5/2) and (6+3.75/2,-3.5+.5/2) ..  node (7T1) [above] {$\scs d$} (144); 
		\draw (1B1) -- node [above,sloped] {\small (PN)} (2T2);
		\draw (1B1) -- node [above,sloped] {\small (PN)} (2T3);	
		\draw (1B1) -- node [above,sloped,pos=.7] {\small (PN)} (15T1);
		\draw (15B1) -- node [below,sloped] {\small (PL)} (2T1);
		\draw (2B2) -- node [above,sloped,near end] {\small (PR)} (3T1);
		\draw (2B3S) -- node [above,sloped] {\small (PL)} (3T2);
		\foreach \x in {1,2}
			\draw (2B1) -- node [above,sloped] {\small (PN)} (4T\x);
		\draw (2B2) -- node [above,sloped, near start] {\small (PN)} (4T2);
		\draw (2B2) -- node [above,sloped,pos=.7] {\small (PN)} (4T3);
		\draw (2B3) -- node [above,sloped,pos=.45] {\small (PN)} (4T1);
		\draw (2B3) -- node [above,sloped] {\small (PN)} (4T3);
		\foreach \x in {2,3}
			\draw (3B1) -- node [below,sloped,pos=.55] {\small (PN)} (5T\x);
		\foreach \x in {1,3}
			\draw (3B2) -- node [above,sloped,pos=.3] {\small (PN)} (5T\x);
		\draw (4B1) -- node [below,sloped] {\small (PL)} (5T1);
		\draw (4B2) -- node [below,sloped] {\small (PR)} (5T2);
		\draw (4B3) -- node [above,sloped,pos=.4] {\small (PB)} (5T3);
		\foreach \x in {1,2,3}
			\draw (5B\x) -- node [below,sloped,pos=.4] {\small (PN)}  (7T1);
		\foreach \x in {1,2,3}
			\draw (4B\x) -- node [below,sloped,pos=.6] {\small (PN)} (6T1);
		\draw (6B1) -- node [above,sloped,] {\small (PB)} (7T1);

\end{tikzpicture}
$$
where to be able depict the interval we have combined the nodes corresponding to the different values of $a',b',c'\in \FF_q^\times$.
\end{example}

\begin{remark}  The conflict poset $\cP$ is related to two natural posets,
\begin{enumerate}
\item[(a)] The poset on $\cA$ given by nesting, or $j\larc{a}k\leq i\larc{b}l$ if $i\leq j<k\leq l$.
\item[(b)]  The poset $\cP$ is a subposet of the poset on $\cM$ given by $\mu\leq \lambda$ if there exists an injective function $\iota:\mu\rightarrow \lambda$ such that $\iota(j\larc{a}k)=i\larc{b}l$, where $i\leq j<k\leq l$.   We use this observation to think of arcs as shrinking as we move down the poset (and potentially disappearing).
\end{enumerate}
\end{remark}

A \emph{path} in $\cP$ is a directed path in the Hasse diagram of the poset $\cP$.  The following proposition justifies the existence of this poset.

\begin{proposition} \label{PosetToCoefficient} Let $\lambda\in \cM_L(q)$ and $\mu\in\cM_K(q)$ with $K\subseteq L$.  
Then $\mu\prec \lambda$ if and only if either
\begin{enumerate}
\item[(a)] $K=L$ and $\chi^{\lambda}=c_\mu^\lambda\chi^{\mu}+\chi'$, where $c_\mu^\lambda\in \ZZ_{>0}$ and $\chi'$ is a character of $U_K(q)$,
\item[(b)] $K\subset L$ and $\Res_{U_K(q)}^{U_L(q)}(\chi^\lambda)=c_\mu^\lambda\chi^\mu +\chi',$ where $c_\mu^\lambda\in \ZZ_{>0}$ and $\chi'$ is a character of $U_K(q)$.
\end{enumerate}
\end{proposition}

\begin{proof}
Following a path in $\cP$ corresponds to resolving conflicts in $\lambda$ using either Theorem \ref{ArcRestriction} or Theorem \ref{ArcTensorProduct}.
\end{proof}

\begin{remark} \hfill

\begin{enumerate}
\item[(a)] It is tempting to assume that $c_\mu^\lambda$ is equal to the number of paths from $\lambda$ to $\mu$ in the Hasse diagram.  However, this turns out to be false.  To get such a result one would need to use a subposet of $\cP$ by essentially specifying a ``canonical order" for resolving conflicts.
\item[(b)] Another way to view the purpose of this paper is as a characterization of  when two set partitions are comparable in $\cP$, or more generally when an arbitrary multiset is greater than a set partition.
\end{enumerate}
\end{remark}

\subsection{Proof of Theorem \ref{MainTheorem}}

This section proves Theorem  \ref{MainTheorem}.  Lemma \ref{MainLemma} is the heart of the proof, allowing us to deduce the appearance of the trivial character in a decomposition from the appearance of another supercharacter.   Unfortunately, the proof is quite technical, and there is probably a more elegant proof  of the main theorem to be found that avoids this lemma.  Lemma \ref{NoHalves} then shows that given $\lambda\in \cS_L(q)$, we can remove many of the arcs of $\lambda$ without changing whether $\One$ appears in the decomposition.  The final proof then uses that related elements in $\cP$ have an implicit injective map from the smaller element to the larger one. 

For $\lambda\in\cM_K(q)$, let $\bar{\lambda}\in \cM_K(q)$ be given by 
$$\bar{\lambda}=\{i\larc{-a}l\mid i\larc{a}l\in \lambda\}.$$
Note that by (\ref{ArcDecomposition}) the conjugate character $\overline{\chi^\lambda}=\chi^{\bar\lambda}$.   The first proposition is a standard result for irreducible characters, but its analogue bears observing for supercharacters.

\begin{proposition}
Let $\lambda,\mu\in\cS_K(q)$ with $K\subseteq \ZZ_{\geq 1}$.  Then
$$\langle\chi^\lambda\otimes\chi^\mu,\One\rangle=\delta_{\lambda\bar{\mu}}q^{|\cC(\lambda)|}.$$
\end{proposition}
\begin{proof}
By the orthogonality of supercharacters,
$$\langle\chi^\lambda\otimes\chi^\mu,\One\rangle=\langle\chi^\lambda,\chi^{\bar{\mu}}\rangle=\delta_{\lambda\bar{\mu}}q^{|\cC(\lambda)|}.\qedhere$$
\end{proof}

The following lemma allows us to deduce the appearance of the trivial character in the decomposition of a restriction from the appearance of a single arc that was not in the original set partition.  The proof of this lemma is the crux of proving the main result.

\begin{lemma} \label{MainLemma}
Let $K\subseteq L\subseteq \ZZ_{\geq 1}$ be finite sets, and let $\lambda\in\cS_L(q)$ .  
\begin{enumerate}
\item[(a)] If there exists $i\in K$ such that $i\larc{}j\notin \lambda$ for all $j>i$, then for all $l>i$ and $a\in\FF_q^\times$,
$$\langle \Res_{U_K(q)}^{U_L(q)}(\chi^{\lambda}),\chi^{i\slarc{a}l}\rangle\neq 0\qquad\text{implies}\qquad \langle \Res_{U_K(q)}^{U_L(q)}(\chi^{\lambda}),\One\rangle\neq 0.$$
\item[(b)] If there exists $l\in K$ such that $k\larc{}l\notin \lambda$ for all $k<l$, then for all $i<l$ and $a\in \FF_q^\times$,
$$\langle \Res_{U_K(q)}^{U_L(q)}(\chi^{\lambda}),\chi^{i\slarc{a}l}\rangle\neq 0\qquad\text{implies}\qquad \langle \Res_{U_K(q)}^{U_L(q)}(\chi^{\lambda}),\One\rangle\neq 0.$$
\end{enumerate}
\end{lemma}

\begin{proof}
(a) Suppose $\langle \Res_{U_K(q)}^{U_L(q)}(\chi^{\lambda}),\chi^{i\slarc{a}l}\rangle\neq 0$ for some $l>i$ and $a\in\FF_q^\times$.  By assumption there is a path from $\lambda$ to $\{i\larc{a}l\}$ in $\cP$.  Since there are injective maps $\iota$ at each step of the path we can track the arc that will eventually become $i\larc{a}l$ in this path. We will refer to this arc as arc $\vphi$ (regardless of its particular endpoints at a given point in the path).   In fact, given a fixed choice of injections, we can track all the original arcs and the conflicts they have with $\vphi$ along the way.   

Since $i\larc{}k\notin \lambda$ for all $k$, at some point arc $\vphi$ must have changed from $i'\larc{a''}l''\mapsto i\larc{a'}l'$ where $i'<i<l'\leq l''$.  At this point in the path we have the multiset $\{i\larc{a'}l'\}\cup \nu$ for some multiset $\nu$.  The arcs that interact with arc $\vphi$ the rest of the way come in three flavors:
\begin{enumerate}
\item[(F1)]  $i\larc{}k$ for some $k\leq l^*$, where $\vphi=i\larc{a^*}l^*$,
\item[(F2)]  $j\larc{}l^*$ for some $i<j$, where $\vphi=i\larc{a^*}l^*$,
\item[(F3)]  $h\larc{}l^*$ for some $h<i$, where $\vphi=i\larc{a^*}l^*$.
\end{enumerate}
Our goal  is to show that we can replace the image  $i\larc{a'}l'$ of $\vphi$ with another arc, without changing where  the other arcs (F1)--(F3) end up.  By being careful, the (F2)--(F3) arcs will not be affected (they act in exactly the same way), but we will need to deal with the (F1) arcs separately, below.

List all the (F2) arcs, $j_1\larc{a_1}l_1,j_2\larc{a_2}l_2,\ldots, j_r\larc{a_r}l_r$, where $j_1\leq j_2\leq \cdots \leq j_r$ and if $j_m=j_{m+1}$, then $l_m\geq l_{m+1}$.    Let $l_{i_1}\geq l_{i_2}\geq \ldots \geq l_{i_m}$ be the longest decreasing subsequence of $(l_1,l_2,\ldots, l_r)$ beginning at $l_{i_1}=l_1$ such that if $l_{i_m}=l_{i_{m+1}}$, then $j_{i_m}<j_{i_{m+1}}$.  At the stage where we sent $i'\larc{a''}l''\mapsto i\larc{a'}l'$, we could have sent 
$$\vphi=i'\larc{a''}l''\mapsto \left\{\begin{array}{ll}  j_1\larc{-a_1^*}l', & \text{if $j_1$ exists, $a_1^*=\wt_{j_1l_1}(\{j_1\larc{a_1}l_1,j_2\larc{a_2}l_2,\ldots, j_r\larc{a_r}l_r\})$,}\\  \emptyset & \text{otherwise.}\end{array}\right.$$ 
If $j_1$ exists, we will continue to let $\vphi$ interact with (F2) and (F3) arcs as before (except throughout we choose the label $-a_1^*$ for $\vphi$), until we get to the conflicts $(R_{j_1l_1},j_1\larc{-a_1^*}l_1)$.  Note that since $l_{i_2}<l_{i_{2}-1},\ldots,l_2$  and arcs shrink as we move down the poset,  the arcs $j_2\larc{a_2}l_2,\ldots, j_{i_{2}-1}\larc{a_{i_{2}-1}}l_{i_{2}-1}$ will already have interacted in the original fashion with our arc $\vphi$ before we resolved the conflicts  $(R_{j_1l_1},j_1\larc{-a_1^*}l_1)$.    We resolve the conflicts $(R_{j_1l_1},j_1\larc{-a_1^*}l_1)$ by sending each arc in $R_{j_1l_1}$ to  wherever it went in the original $(R_{j_1l_1},i\larc{a'}l_1)$-conflicts, and arc 
$$\vphi=j_1\larc{-a_1^*}l_1\mapsto\left\{\begin{array}{ll}  j_{i_2}\larc{-a_{i_2}^*}l_1 & \text{if $i_2$ exists, $a_{i_2}^*=\wt_{j_{i_2}l_{i_2}}(\{j_{i_2}\larc{a_{i_2}}l_{i_2},\ldots, j_r\larc{a_r}l_r\})$,}\\ \emptyset, & \text{otherwise.}\end{array}\right.$$ 
 We iterate this process and  all (F2) arcs act the same way as before, but there is no arc $\vphi$ in the end. 

Note that nothing has changed for the (F3) arcs in this process.  An example is probably instructive at this point.  Suppose that in the original path, $\vphi$ started as $i\larc{a}j_6$, and the path continues as 
$$
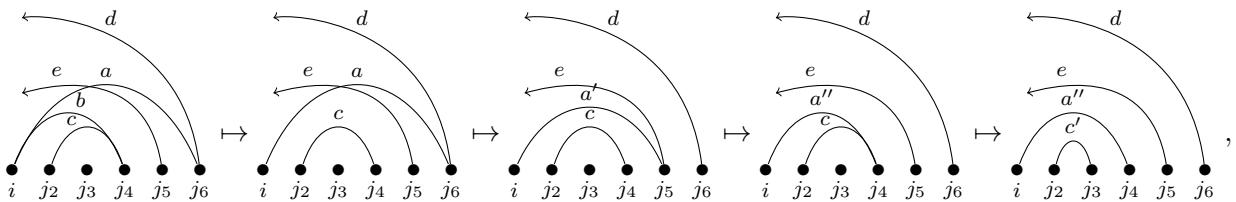
\begin{tikzpicture}[baseline=.4cm]
	\foreach \x in {1,...,6} 
		\node (\x) at (\x/2,0) [inner sep=0pt] {$\bullet$};
	\foreach \x in {2,...,6} 
		\node at (\x/2,-.25) {$\scs j_{\x}$};
	\node at (1/2,-.25) {$\scs i$};
	\node (11) at (1/2,1) {};
	\node (12) at (1/2,2) {};
	\draw (1) .. controls (2.25/2,1.5) and (4.75/2,1.5) .. node [above=-1pt] {$\scs a$} (6); 
	\draw (1) .. controls (1.75/2,1) and (3.25/2,1) .. node [above=-1pt,pos=.6] {$\scs b$} (4); 
	\draw (2) .. controls (2.5/2,.75) and (3.5/2,.75) .. node [above=-1pt,pos=.3] {$\scs c$} (4); 
	\draw[->] (6) .. controls (5.5/2,2) and (1.75/2,2.1) .. node [above] {$\scs d$} (12); 
	\draw[->] (5) .. controls (4.5/2,1.5) and (1.75/2,1.1) .. node [above,pos=.7] {$\scs e$} (11); 
\end{tikzpicture}
\hspace{-.1cm}\mapsto\hspace{-.1cm}
\begin{tikzpicture}[baseline=.4cm]
	\foreach \x in {1,...,6} 
		\node (\x) at (\x/2,0) [inner sep=0pt] {$\bullet$};
	\foreach \x in {2,...,6} 
		\node at (\x/2,-.25) {$\scs j_{\x}$};
	\node at (1/2,-.25) {$\scs i$};
	\node (11) at (1/2,1) {};
	\node (12) at (1/2,2) {};
	\draw (1) .. controls (2.25/2,1.5) and (4.75/2,1.5) .. node [above=-1pt] {$\scs a$} (6); 
	\draw (2) .. controls (2.5/2,.75) and (3.5/2,.75) .. node [above=-1pt,] {$\scs c$} (4); 
	\draw[->] (6) .. controls (5.5/2,2) and (1.75/2,2.1) .. node [above] {$\scs d$} (12); 
	\draw[->] (5) .. controls (4.5/2,1.5) and (1.75/2,1.1) .. node [above,pos=.7] {$\scs e$} (11); 
\end{tikzpicture}
\hspace{-.1cm}\mapsto\hspace{-.1cm}
\begin{tikzpicture}[baseline=.4cm]
	\foreach \x in {1,...,6} 
		\node (\x) at (\x/2,0) [inner sep=0pt] {$\bullet$};
	\foreach \x in {2,...,6} 
		\node at (\x/2,-.25) {$\scs j_{\x}$};
	\node at (1/2,-.25) {$\scs i$};
	\node (11) at (1/2,1) {};
	\node (12) at (1/2,2) {};
	\draw (1) .. controls (2/2,1.1) and (4/2,1.1) .. node [above=-1pt] {$\scs a'$} (5); 
	\draw (2) .. controls (2.5/2,.75) and (3.5/2,.75) .. node [above=-1pt,] {$\scs c$} (4); 
	\draw[->] (6) .. controls (5.5/2,2) and (1.75/2,2.1) .. node [above] {$\scs d$} (12); 
	\draw[->] (5) .. controls (4.5/2,1.5) and (1.75/2,1.1) .. node [above,pos=.7] {$\scs e$} (11); 
\end{tikzpicture}
\hspace{-.1cm}\mapsto\hspace{-.1cm}
\begin{tikzpicture}[baseline=.4cm]
	\foreach \x in {1,...,6} 
		\node (\x) at (\x/2,0) [inner sep=0pt] {$\bullet$};
	\foreach \x in {2,...,6} 
		\node at (\x/2,-.25) {$\scs j_{\x}$};
	\node at (1/2,-.25) {$\scs i$};
	\node (11) at (1/2,1) {};
	\node (12) at (1/2,2) {};
	\draw (1) .. controls (1.75/2,1) and (3.25/2,1) .. node [above=-1pt] {$\scs a''$} (4); 
	\draw (2) .. controls (2.5/2,.75) and (3.5/2,.75) .. node [above=-1pt,pos=.3] {$\scs c$} (4); 
	\draw[->] (6) .. controls (5.5/2,2) and (1.75/2,2.1) .. node [above] {$\scs d$} (12); 
	\draw[->] (5) .. controls (4.5/2,1.5) and (1.75/2,1.1) .. node [above,pos=.7] {$\scs e$} (11); 
\end{tikzpicture}
\hspace{-.1cm}\mapsto \hspace{-.1cm}
\begin{tikzpicture}[baseline=.4cm]
	\foreach \x in {1,...,6} 
		\node (\x) at (\x/2,0) [inner sep=0pt] {$\bullet$};
	\foreach \x in {2,...,6} 
		\node at (\x/2,-.25) {$\scs j_{\x}$};
	\node at (1/2,-.25) {$\scs i$};
	\node (11) at (1/2,1) {};
	\node (12) at (1/2,2) {};
	\draw (1) .. controls (1.75/2,1) and (3.25/2,1) .. node [above=-1pt] {$\scs a''$} (4); 
	\draw (2) .. controls (2.25/2,.5) and (2.75/2,.5) .. node [above=-1pt] {$\scs c'$} (3); 
	\draw[->] (6) .. controls (5.5/2,2) and (1.75/2,2.1) .. node [above] {$\scs d$} (12); 
	\draw[->] (5) .. controls (4.5/2,1.5) and (1.75/2,1.1) .. node [above,pos=.7] {$\scs e$} (11); 
\end{tikzpicture},
$$
where $i\larc{b}j_4$ is (F1), $j_2\larc{c}j_4$ is (F2) and the other two arcs ending at $j_5$ and $j_6$ are (F3).  The above algorithm suggests that if we had instead mapped $\vphi$ to $j_2\larc{-c} j_6$ instead of $i\larc{a}j_6$, we obtain
$$
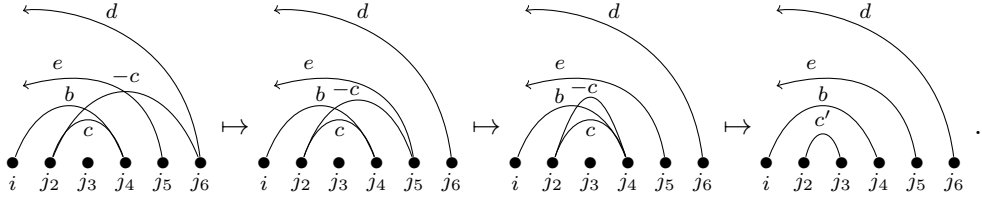
\begin{tikzpicture}[baseline=.4cm]
	\foreach \x in {1,...,6} 
		\node (\x) at (\x/2,0) [inner sep=0pt] {$\bullet$};
	\foreach \x in {2,...,6} 
		\node at (\x/2,-.25) {$\scs j_{\x}$};
	\node at (1/2,-.25) {$\scs i$};
	\node (11) at (1/2,1) {};
	\node (12) at (1/2,2) {};
	\draw (2) .. controls (3/2,1.25) and (5/2,1.25) .. node [above=-1pt] {$\scs -c$} (6); 
	\draw (1) .. controls (1.75/2,1) and (3.25/2,1) .. node [above=-2pt] {$\scs b$} (4); 
	\draw (2) .. controls (2.5/2,.75) and (3.5/2,.75) .. node [below=-1pt] {$\scs c$} (4); 
	\draw[->] (6) .. controls (5.5/2,2) and (1.75/2,2.1) .. node [above] {$\scs d$} (12); 
	\draw[->] (5) .. controls (4.5/2,1.5) and (1.75/2,1.1) .. node [above,pos=.7] {$\scs e$} (11); 
\end{tikzpicture}
\hspace{-.1cm}\mapsto\hspace{-.1cm}
\begin{tikzpicture}[baseline=.4cm]
	\foreach \x in {1,...,6} 
		\node (\x) at (\x/2,0) [inner sep=0pt] {$\bullet$};
	\foreach \x in {2,...,6} 
		\node at (\x/2,-.25) {$\scs j_{\x}$};
	\node at (1/2,-.25) {$\scs i$};
	\node (11) at (1/2,1) {};
	\node (12) at (1/2,2) {};
	\draw (2) .. controls (2.75/2,1.1) and (4.25/2,1.1) .. node [above=-2pt,pos=.4] {$\scs -c$} (5); 
	\draw (1) .. controls (1.75/2,1) and (3.25/2,1) .. node [above=-2pt] {$\scs b$} (4); 
	\draw (2) .. controls (2.5/2,.75) and (3.5/2,.75) .. node [below=-1pt] {$\scs c$} (4); 
	\draw[->] (6) .. controls (5.5/2,2) and (1.75/2,2.1) .. node [above] {$\scs d$} (12); 
	\draw[->] (5) .. controls (4.5/2,1.5) and (1.75/2,1.1) .. node [above,pos=.7] {$\scs e$} (11); 
\end{tikzpicture}
\hspace{-.1cm}\mapsto\hspace{-.1cm}
\begin{tikzpicture}[baseline=.4cm]
	\foreach \x in {1,...,6} 
		\node (\x) at (\x/2,0) [inner sep=0pt] {$\bullet$};
	\foreach \x in {2,...,6} 
		\node at (\x/2,-.25) {$\scs j_{\x}$};
	\node at (1/2,-.25) {$\scs i$};
	\node (11) at (1/2,1) {};
	\node (12) at (1/2,2) {};
	\draw (2) .. controls (2.75/2,1.15) and (3.25/2,1.15) .. node [above=-2pt,pos=.4] {$\scs -c$} (4); 
	\draw (1) .. controls (1.75/2,1) and (3.25/2,1) .. node [above=-2pt,pos=.4] {$\scs b$} (4); 
	\draw (2) .. controls (2.5/2,.75) and (3.5/2,.75) .. node [below=-1pt] {$\scs c$} (4); 
	\draw[->] (6) .. controls (5.5/2,2) and (1.75/2,2.1) .. node [above] {$\scs d$} (12); 
	\draw[->] (5) .. controls (4.5/2,1.5) and (1.75/2,1.1) .. node [above,pos=.7] {$\scs e$} (11); 
\end{tikzpicture}
\hspace{-.1cm}\mapsto\hspace{-.1cm}
\begin{tikzpicture}[baseline=.4cm]
	\foreach \x in {1,...,6} 
		\node (\x) at (\x/2,0) [inner sep=0pt] {$\bullet$};
	\foreach \x in {2,...,6} 
		\node at (\x/2,-.25) {$\scs j_{\x}$};
	\node at (1/2,-.25) {$\scs i$};
	\node (11) at (1/2,1) {};
	\node (12) at (1/2,2) {};
	\draw (1) .. controls (1.75/2,1) and (3.25/2,1) .. node [above=-2pt] {$\scs b$} (4); 
	\draw (2) .. controls (2.25/2,.5) and (2.75/2,.5) .. node [above=-1pt] {$\scs c'$} (3); 
	\draw[->] (6) .. controls (5.5/2,2) and (1.75/2,2.1) .. node [above] {$\scs d$} (12); 
	\draw[->] (5) .. controls (4.5/2,1.5) and (1.75/2,1.1) .. node [above,pos=.7] {$\scs e$} (11); 
\end{tikzpicture}.
$$
Note that the (F2) and (F3) arcs act the same as before.  However, we are left with the unwanted (F1) arc $i\larc{b}j_4$.

Thus, in making the change $i'\larc{a''}l''\mapsto j_1\larc{-a_1^*}l'$, the (F1) conflicts cannot resolve in the same way.  Suppose $i\larc{b}k$ is a type (F1) conflict, and call it arc $\tau$.  Since $i\larc{}k\notin \lambda$ for any $k$, $\tau$ must have come from some $i'\larc{b'}k'$; we will use this to send $\tau$ to a different arc in a way similar to how we dealt with $\vphi$ above.   The arcs that $\tau$ interacts with the rest of way come in four guises:
\begin{enumerate}
\item[(G1)] $i\larc{c}j$ with $j\leq k^*$, where $\tau=i\larc{b^*} k^*$,
\item[(G2)] $i^*\larc{c}j$ with $i^*>i$, $j\leq k^*$, where $\tau=i^*\larc{b^*}k^*$,
\item[(G3)]  $j\larc{c}k^*$ with $j>i^*$, where $\tau=i^*\larc{b^*}k^*$,
\item[(G4)] $i^*\larc{c}l$ with  $k^*<l$ or  $h\larc{c}k^*$ with $h<i^*$, where $\tau=i^*\larc{b^*}k^*$,
\end{enumerate}
WLOG we may ignore the (G1) type conflicts.  That is, since they have the same properties as $\tau$, we may apply the same strategy to them as we do to $\tau$.  

Let $i_1\larc{b_1}k_1,i_2\larc{b_2}k_2,\ldots, i_\ell\larc{b_\ell}k_\ell$ be the set of (G2) and (G3) type conflict arcs, where $i_1\leq i_2\leq \cdots \leq i_\ell$ and if $i_m=i_{m+1}$, then $k_{m+1}\leq k_m$.   Since $i\larc{}k\notin \lambda$ for any $k$, $\tau$ must have come from some $i'\larc{b'}k'$.  Instead of sending $i'\larc{b'}k'\mapsto i\larc{b}k$, send 
 $$ \vphi=i'\larc{b'}k' \mapsto \left\{\begin{array}{ll} \emptyset, & \text{if $i_1$ does not exist,}\\
i_1\larc{b^*}k , & \text{if $i_1\larc{b_1} k_1$ is type (G2), $(i\larc{a^*}l^*,i\larc{b}k)$ originally mapped to $(i\larc{a^*}l^*,i_1\larc{b^*}k)$,}\\
i_1\larc{-b_1^*}k , & \text{if $i_1\larc{b_1} k_1$ is type (G3), $b_1^*=\wt_{i_1l_1}(\{i_1\larc{b_1}k_1,i_2\larc{b_2}k_2,\ldots, i_\ell\larc{b_\ell}k_\ell\})$.}
 \end{array}\right.$$
In the first case, we are done.  In the second case, the remaining conflicts can be resolved in the same way as before (we just skipped a step).  In the third case, we will eventually have the conflicts $(R_{i_1k_1},i_1\larc{-b_1^*}k_1)$ which we resolve by sending each arc $i_1\larc{}k_1$ in $R_{i_1k_1}$  to wherever it went in the original conflict $(R_{i_1k_1},i^*\larc{b^*}k^*)$, and we iterate the process for $i_1\larc{-b_1} k_1$ with the remaining conflicts $\{i_m\larc{b_m}k_m\mid m>1,k_m\leq k_1\}$.

The type (G4) conflicts interact with $\tau$ the same way as before.

Thus, by changing our choices slightly, the arcs all behave in a similar way except that we end up without the $i\larc{a}l$ arc at the end.

(b) The proof is symmetric.
\end{proof}

Given $\lambda\in \cS_L(q)$, the following lemma allows us to discard all the arcs that have exactly one endpoint in $K$ without changing whether the resulting restriction has a trivial character in its decomposition.

\begin{lemma} \label{NoHalves}
Let  $K\subseteq L\subseteq \ZZ_{\geq 1}$ be finite sets and let $\lambda\in \cS_L(q)$.   Suppose there exists $i\larc{a}l\in \lambda$ such that $|\{i,l\}\cap K|=1$.  Then
$$\langle\Res_{U_K(q)}^{U_L(q)}(\chi^\lambda),\One\rangle=0\qquad\text{if and only if}\qquad \langle\Res_{U_K(q)}^{U_L(q)}(\chi^{\lambda-\{i\slarc{a}l\}}),\One\rangle=0.$$
\end{lemma}
\begin{proof}
WLOG assume that $l\notin K$.  Suppose $\langle\Res_{U_K(q)}^{U_L(q)}(\chi^{\lambda-\{i\slarc{a}l\}}),\One\rangle\neq 0$. Then by Theorem \ref{ArcRestriction},
\begin{align*}
\langle\Res_{U_K}^{U_L}(\chi^\lambda),&\One\rangle
=\langle\Res_{U_K}^{U_L}(\chi^{\lambda-\{i\slarc{a}l\}})\otimes \Res_{U_K}^{U_L}(\chi^{i\slarc{a}l}),\One\rangle\\
&=q^{r_K^L(i\slarc{a}l)}\bigg(\langle\Res_{U_K}^{U_L}(\chi^{\lambda-\{i\slarc{a}l\}})\otimes\One,\One\rangle+\hspace{-.35cm}\sum_{i<j<l\atop j\in K, b\in\FF_q^\times} \hspace{-.25cm}\langle\Res_{U_K}^{U_L}(\chi^{\lambda-\{i\slarc{a}l\}})\otimes\chi^{i\slarc{b}j},\One\rangle\bigg)\\ 
&>0.
\end{align*}

Conversely, suppose $\langle\Res_{U_K}^{U_L}(\chi^\lambda),\One\rangle\neq 0$.  Then 
$$q^{r_K^L(i\slarc{a}l)}\bigg(\langle\Res_{U_K}^{U_L}(\chi^{\lambda-\{i\slarc{a}l\}})\otimes\One,\One\rangle+\hspace{-.35cm}\sum_{i<j<l\atop j\in K, b\in\FF_q^\times} \hspace{-.25cm}\langle\Res_{U_K}^{U_L}(\chi^{\lambda-\{i\slarc{a}l\}})\otimes\chi^{i\slarc{b}j},\One\rangle\bigg)\neq 0.$$
Thus, either $\langle \Res_{U_K}^{U_L}(\chi^{\lambda-\{i\slarc{a}l\}}),\One\rangle\neq 0$ or there exists $i<j<l$ and $b\in \FF_q^\times$ such that
$$\langle \Res_{U_K}^{U_L}(\chi^{\lambda-\{i\slarc{a}l\}}),\chi^{i\slarc{b}j}\rangle\neq 0.$$
However, in the latter case Lemma \ref{MainLemma} implies that $\langle \Res_{U_K}^{U_L}(\chi^{\lambda-\{i\slarc{a}l\}}),\One\rangle\neq 0.$
\end{proof}

We can now prove the main theorem.

\begin{proof}[Proof of Theorem \ref{MainTheorem}] 
By Lemma \ref{NoHalves}, we may assume that for all $i\larc{a}l\in \lambda$, either $i,l\in K$ or  $i,l\notin K$ (else, we could remove the other arcs without affecting the whether the coefficient of $\One$ is nonzero).  

Suppose that $\Gamma_K(\lambda)$ has an appropriate matching from $V_\bullet$ to $V_\circ$.   Then
$$\Res_{U_K}^{U_L}(\chi^\lambda)=\bigotimes_{\text{Matched pairs}\atop (j\slarc{b}k,i\slarc{a}l)\in V_\bullet\times V_\circ} \Res_{U_K}^{U_L}(\chi^{\{j\slarc{b}k,i\slarc{a}l\}})\otimes \bigotimes_{\text{Unmatched}\atop i\slarc{a}l\in V_\circ} \Res_{U_K}^{U_L} (\chi^{i\slarc{a}l}).
$$ 
For each matched pair, by Theorem \ref{ArcTensorProduct},
\begin{align*}
\Res_{U_K}^{U_L}(\chi^{\{j\slarc{b}k,i\slarc{a}l\}})
&=\Res_{U_K}^{U_L}(\chi^{j\slarc{b}k})\otimes \Res_{U_K}^{U_L}(\chi^{i\slarc{a}l})\\
&=q^{|\{j<j'<k\mid j'\notin K\}|} \chi^{j\slarc{b}k}\otimes \bigg(c_\one^{i\slarc{a}l}\One+\sum_{i<i'<l'<l\atop i',l'\in K,a'\in \FF_q^\times} c_{i'\slarc{a'}l'}^{i\slarc{a}l}\chi^{i'\slarc{a'}l'}\bigg)
\end{align*}
where $c_{i'\slarc{a'}l'}^{i\slarc{a}l},c_\one^{i\slarc{a}l}\in \ZZ_{>0}$.  Since $i<j<k<l$, $\chi^{j\slarc{-b}k}$ is a nonzero summand of $\Res_{U_K}^{U_L}(\chi^{i\slarc{a}l})$.  Since $\One$ is a nonzero summand of $\chi^{\{j\slarc{b}k,j\slarc{-b}k\}}$,
$$\langle \bigotimes_{\text{Matched pairs}\atop (j\slarc{b}k,i\slarc{a}l)\in V_\bullet\times V_\circ} \Res_{U_K}^{U_L}(\chi^{\{j\slarc{b}k,i\slarc{a}l\}}),\One\rangle\neq 0.$$
By Theorem \ref{ArcRestriction}, each unmatched term $\Res_{U_K}^{U_L} (\chi^{i\slarc{a}l})$ has $\One$ as a nonzero summand, so 
$$\langle\bigotimes_{\text{Unmatched}\atop i\slarc{a}l\in V_\circ} \Res_{U_K}^{U_L} (\chi^{i\slarc{a}l}),\One\rangle\neq 0.$$
Thus, 
$$\langle\Res_{U_K}^{U_L}(\chi^\lambda),\One\rangle\neq 0.$$

Conversely, suppose that 
$$\langle\Res_{U_K}^{U_L}(\chi^\lambda),\One\rangle\neq 0.$$
Write $\lambda=\mu\cup \nu$, where
\begin{align*}
\mu &= V_\circ = \{i\larc{a}l\in \lambda\mid i,l\notin K\}\\
\nu &=V_\bullet = \{j\larc{b}k\in \lambda\mid j,k\in K\}.
\end{align*}
By Theorem \ref{ArcRestriction} and (\ref{AlreadyDownArcsRestriction}),
\begin{equation*}
\Res_{U_K}^{U_L}(\chi^\lambda)=\Res_{U_K}^{U_L}(\chi^\mu)\otimes\Res_{U_K}^{U_L}(\chi^\nu)= \Res_{U_K}^{U_L}(\chi^\mu)\otimes q^{r_K^L(\lambda)} \chi^\nu,
\end{equation*}
since $\nu\in \cS_K(q)\cap\cS_L(q)$.  Since $\langle\Res_{U_K}^{U_L}(\chi^\lambda),\One\rangle\neq 0$, 
$$\langle \Res_{U_K}^{U_L}(\chi^\mu), \chi^{\bar{\nu}}\rangle=\langle \Res_{U_K}^{U_L}(\chi^\mu)\otimes \chi^{\nu},\One\rangle\neq 0.$$
However, Proposition \ref{PosetToCoefficient} implies $\bar{\nu}\preceq \lambda$, so there exists an injective function $\iota:\nu\rightarrow \mu$, giving the appropriate matching for $\Gamma_K(\lambda)$.
\end{proof}

\section{Tensor products and general restriction coefficients}

Theorem \ref{MainTheorem} in fact is sufficiently strong that analogous statements can be made for the coefficients of arbitrary supercharacters in the decomposition of tensor products and restrictions.  This section begins by developing the appropriate generalization to the graph $\Gamma_K(\lambda)$.  We then prove Theorem \ref{TensorResult} for tensor products and Theorem \ref{RestrictionResult} for restrictions.  Along the way, Theorem \ref{StraighteningTheorem} describes how characters corresponding to multisets are the same as restrictions from certain set partitions (up to a scalar multiple).

\subsection{A generalized bipartite graph} \label{GeneralGraphSection}

Given $\lambda\in \cM(q)$, perturb the arcs such that they stack on top of one-another in the following fashion.  
\begin{enumerate}
\item[(TL)]  If $i\larc{a}j,i\larc{b}k\in \lambda$ with $j<k$, then the left endpoint of $i\larc{b}k$ is above the left endpoint of $i\larc{a}j$,
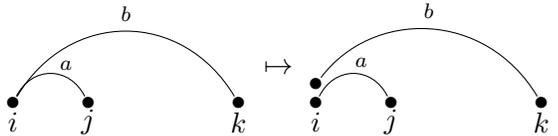
$$\begin{tikzpicture}[baseline=.4cm]
	\node (i) at (0,0) [inner sep=0pt] {$\bullet$};
	\node at (0,-.25) {$i$};
	\node (j) at (1,0) [inner sep=0pt] {$\bullet$};
	\node at (1,-.25) {$j$};
	\node (k) at (3,0) [inner sep=0pt] {$\bullet$};
	\node at (3,-.25) {$k$};
	\draw (i) .. controls (.25,.5) and (.75,.5) .. node [above=-1pt,pos=.7] {$\scs a$}  (j); 
	\draw (i) .. controls (.75,1.25) and (2.25,1.25) .. node [above] {$\scs b$}  (k);
\end{tikzpicture} 
\mapsto
\begin{tikzpicture}[baseline=.4cm]
	\node (i) at (0,0) [inner sep=0pt] {$\bullet$};
	\node (i1) at (0,.25) [inner sep=0pt] {$\bullet$};
	\node at (0,-.25) {$i$};
	\node (j) at (1,0) [inner sep=0pt] {$\bullet$};
	\node at (1,-.25) {$j$};
	\node (k) at (3,0) [inner sep=0pt] {$\bullet$};
	\node at (3,-.25) {$k$};
	\draw (i) .. controls (.25,.5) and (.75,.5) .. node [above=-1pt,pos=.6] {$\scs a$}  (j); 
	\draw (i1) .. controls (.75,1.25) and (2.25,1.25) .. node [above] {$\scs b$}  (k);
\end{tikzpicture}. 
$$
\item[(TR)] If $i\larc{a}k,j\larc{b}k\in\lambda$ with $i<j$, then the right endpoint of $i\larc{a}k$ is above the right endpoint of $j\larc{b}k$,
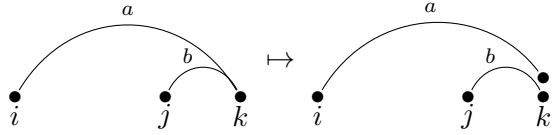
$$\begin{tikzpicture}[baseline=.4cm]
	\node (i) at (0,0) [inner sep=0pt] {$\bullet$};
	\node at (0,-.25) {$i$};
	\node (j) at (2,0) [inner sep=0pt] {$\bullet$};
	\node at (2,-.25) {$j$};
	\node (k) at (3,0) [inner sep=0pt] {$\bullet$};
	\node at (3,-.25) {$k$};
	\draw (j) .. controls (2.25,.5) and (2.75,.5) .. node [above=-1pt,pos=.3] {$\scs b$}  (k); 
	\draw (i) .. controls (.75,1.25) and (2.25,1.25) .. node [above] {$\scs a$}  (k);
\end{tikzpicture} 
\mapsto
\begin{tikzpicture}[baseline=.4cm]
	\node (i) at (0,0) [inner sep=0pt] {$\bullet$};
	\node at (0,-.25) {$i$};
	\node (j) at (2,0) [inner sep=0pt] {$\bullet$};
	\node at (2,-.25) {$j$};
	\node (k) at (3,0) [inner sep=0pt] {$\bullet$};
	\node (k1) at (3,.25) [inner sep=0pt] {$\bullet$};
	\node at (3,-.25) {$k$};
	\draw (j) .. controls (2.25,.5) and (2.75,.5) .. node [above=-1pt,pos=.3] {$\scs b$}  (k); 
	\draw (i) .. controls (.75,1.25) and (2.25,1.25) .. node [above] {$\scs a$}  (k1);
\end{tikzpicture} .$$
\item[(TB)] If $|R_{jk}|>1$, then
$$R_{jk}=
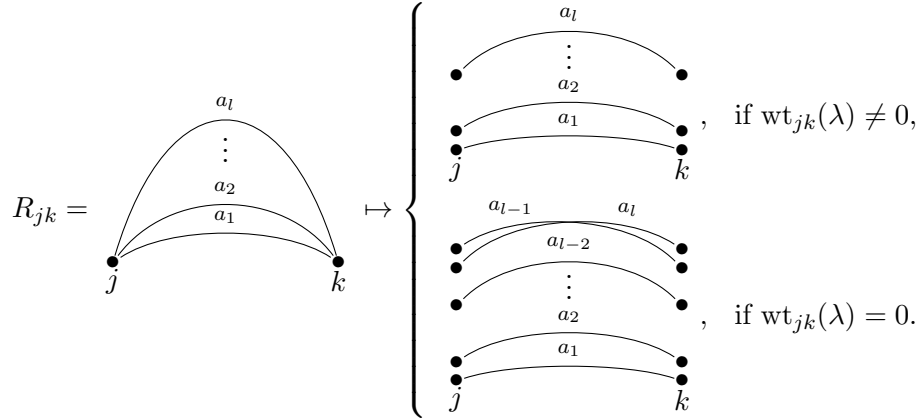
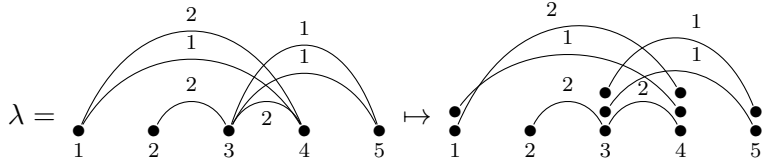
\begin{tikzpicture}[baseline=.6cm]
	\node (i) at (0,0) [inner sep=0pt] {$\bullet$};
	\node at (0,-.25) {$j$};
	\node (j) at (3,0) [inner sep=0pt] {$\bullet$};
	\node at (3,-.25) {$k$};
	\node at (1.5,1.6) {$\scs\vdots$};
	\draw (i) .. controls (.75,.5) and (2.25,.5) .. node [above] {$\scs a_1$}  (j); 
	\draw (i) .. controls (.75,1) and (2.25,1) .. node [above] {$\scs a_2$}  (j);
	\draw (i) .. controls (.75,2.5) and (2.25,2.5) .. node [above] {$\scs a_l$}  (j); 
\end{tikzpicture} \mapsto
 \left\{
\begin{array}{ll} 
\begin{tikzpicture}[baseline=.6cm]
	\foreach \x in {1,2,5} 
		\node (i\x) at (0,\x/4) [inner sep=0pt] {$\bullet$};
	\node at (0,0) {$j$};
	\foreach \y in {1,2,5} 
		\node (j\y) at (3,\y/4) [inner sep=0pt] {$\bullet$};
	\node at (3,0) {$k$};
	\node at (1.5,1.6) {$\scs\vdots$};
	\draw (i1) .. controls (.75,.5) and (2.25,.5) .. node [above] {$\scs a_1$}  (j1); 
	\draw (i2) .. controls (.75,1) and (2.25,1) .. node [above] {$\scs a_2$}  (j2);
	\draw (i5) .. controls (.75,2) and (2.25,2) .. node [above] {$\scs a_l$}  (j5); 
\end{tikzpicture} 
, & \text{if $\wt_{jk}(\lambda)\neq 0$,}\\
\begin{tikzpicture}[baseline=1cm]
	\foreach \x in {1,2,5,7,8} 
		\node (i\x) at (0,\x/4) [inner sep=0pt] {$\bullet$};
	\node at (0,0) {$j$};
	\foreach \y in {1,2,5,7,8} 
		\node (j\y) at (3,\y/4) [inner sep=0pt] {$\bullet$};
	\node at (3,0) {$k$};
	\node at (1.5,1.6) {$\scs\vdots$};
	\draw (i1) .. controls (.75,.5) and (2.25,.5) .. node [above] {$\scs a_1$}  (j1); 
	\draw (i2) .. controls (.75,1) and (2.25,1) .. node [above] {$\scs a_2$}  (j2);
	\draw (i5) .. controls (.75,2) and (2.25,2) .. node [above] {$\scs a_{l-2}$}  (j5); 
	\draw (i7) .. controls (.75,2.5) and (2.25,2.5) .. node [above,near end] {$\scs a_l$}  (j8); 
	\draw (i8) .. controls (.75,2.5) and (2.25,2.5) .. node [above,near start] {$\scs a_{l-1}$}  (j7); 	
\end{tikzpicture} 
, & \text{if $\wt_{jk}(\lambda)= 0$.}
\end{array}\right.
$$
\end{enumerate}
\begin{example}
For $q=3$,
$$\lambda=
\begin{tikzpicture}[baseline=.1cm]
	\foreach \x in {1,...,5} 
		\node (\x) at (\x,0) [inner sep=0pt] {$\bullet$};
	\foreach \x in {1,...,5} 
		\node at (\x,-.25) {$\scs\x$};
	\draw (1) .. controls (1.75,1.25) and (3.25,1.25) ..  node [auto] {$\scs 1$} (4); 
	\draw (1) .. controls (1.75,1.75) and (3.25,1.75) .. node [auto] {$\scs 2$} (4); 
	\draw (2) .. controls (2.25,.5) and (2.75,.5) .. node [auto] {$\scs 2$} (3); 
	\draw (3) .. controls (3.25,.5) and (3.75,.5) .. node [auto,swap] {$\scs 2$} (4); 
	\draw (3) .. controls (3.5,1) and (4.5,1) .. node [auto] {$\scs 1$} (5); 
	\draw (3) .. controls (3.5,1.5) and (4.5,1.5) .. node [auto] {$\scs 1$}  (5);
\end{tikzpicture}
\mapsto
\begin{tikzpicture}[baseline=.1cm]
	\foreach \x in {1,...,5} 
		\node (\x) at (\x,0) [inner sep=0pt] {$\bullet$};
	\foreach \z in {1,3,4,5}
		\node (\z1) at (\z,1/4) [inner sep=0pt] {$\bullet$};
	\foreach \w in {3,4}
		\node (\w2) at (\w,1/2) [inner sep=0pt] {$\bullet$};
	\foreach \x in {1,...,5} 
		\node at (\x,-.25) {$\scs\x$};
	\draw (11) .. controls (1.75,1.25) and (3.25,1.25) ..  node [auto] {$\scs 1$} (41); 
	\draw (1) .. controls (1.75,1.75) and (3.25,1.75) .. node [auto] {$\scs 2$} (42); 
	\draw (2) .. controls (2.25,.5) and (2.75,.5) .. node [auto] {$\scs 2$} (3); 
	\draw (3) .. controls (3.25,.5) and (3.75,.5) .. node [above=-1.5pt] {$\scs 2$} (4); 
	\draw (31) .. controls (3.5,1) and (4.5,1) .. node [auto] {$\scs 1$} (5); 
	\draw (32) .. controls (3.5,1.5) and (4.5,1.5) .. node [auto] {$\scs 1$}  (51);
\end{tikzpicture}
$$
\end{example}

Define a labeling function $\Lambda_K:\lambda\rightarrow \{(\bullet,\bullet),(\circ,\bullet),(\bullet,\circ),(\circ,\circ)\}$, given by
$$\Lambda_K(j\larc{b}k)=(\Lambda_K^L(j\larc{b}k),\Lambda_K^R(j\larc{b}k)),$$
where 
\begin{equation} \label{LabellingRules}
\begin{split}
\Lambda_K^L(j\larc{b}k)&=\left\{\begin{array}{ll} \circ & \text{if $j\notin K$ or $j\larc{b}k$ starts below an arc starting at $j$,}\\ \bullet, &\text{otherwise.}\end{array}\right.\\
\Lambda_K^R(j\larc{b}k)&=\left\{\begin{array}{ll} \circ & \text{if $k\notin K$ or $j\larc{b}k$ ends below an arc ending at $k$,}\\ \bullet, &\text{otherwise.}\end{array}\right.
\end{split}
\end{equation}
In the above example, 
$$\Lambda_{\{1,2,3,4,5\}}(\lambda)= 
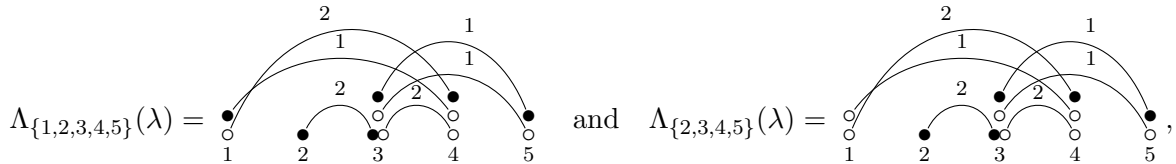
\begin{tikzpicture}[baseline=.1cm]
	\foreach \x in {1,4,5} 
		\node (\x) at (\x,0) [inner sep=0pt] {$\circ$};
	\node (3) at (3.07,0) [inner sep=0pt] {$\circ$};
	\node (2) at (2,0) [inner sep=0pt] {$\bullet$};
	\node (25) at (2.93,0) [inner sep=0pt] {$\bullet$};
	\foreach \z in {1,5}
		\node (\z1) at (\z,1/4) [inner sep=0pt] {$\bullet$};
	\foreach \z in {3,4}
		\node (\z1) at (\z,1/4) [inner sep=0pt] {$\circ$};		
	\foreach \w in {3,4}
		\node (\w2) at (\w,1/2) [inner sep=0pt] {$\bullet$};
	\foreach \x in {1,...,5} 
		\node at (\x,-.25) {$\scs\x$};
	\draw (11) .. controls (1.75,1.25) and (3.25,1.25) ..  node [auto] {$\scs 1$} (41); 
	\draw (1) .. controls (1.75,1.75) and (3.25,1.75) .. node [auto] {$\scs 2$} (42); 
	\draw (2) .. controls (2.25,.5) and (2.75,.5) .. node [auto] {$\scs 2$} (25); 
	\draw (3) .. controls (3.25,.5) and (3.75,.5) .. node [above=-1.5pt] {$\scs 2$} (4); 
	\draw (31) .. controls (3.5,1) and (4.5,1) .. node [auto] {$\scs 1$} (5); 
	\draw (32) .. controls (3.5,1.5) and (4.5,1.5) .. node [auto] {$\scs 1$}  (51);
\end{tikzpicture} \quad\text{and}\quad \Lambda_{\{2,3,4,5\}}(\lambda)= \begin{tikzpicture}[baseline=.1cm]
	\foreach \x in {1,4,5} 
		\node (\x) at (\x,0) [inner sep=0pt] {$\circ$};
	\node (3) at (3.07,0) [inner sep=0pt] {$\circ$};
	\node (2) at (2,0) [inner sep=0pt] {$\bullet$};
	\node (25) at (2.93,0) [inner sep=0pt] {$\bullet$};
	\foreach \z in {5}
		\node (\z1) at (\z,1/4) [inner sep=0pt] {$\bullet$};
	\foreach \z in {1,3,4}
		\node (\z1) at (\z,1/4) [inner sep=0pt] {$\circ$};		
	\foreach \w in {3,4}
		\node (\w2) at (\w,1/2) [inner sep=0pt] {$\bullet$};
	\foreach \x in {1,...,5} 
		\node at (\x,-.25) {$\scs\x$};
	\draw (11) .. controls (1.75,1.25) and (3.25,1.25) ..  node [auto] {$\scs 1$} (41); 
	\draw (1) .. controls (1.75,1.75) and (3.25,1.75) .. node [auto] {$\scs 2$} (42); 
	\draw (2) .. controls (2.25,.5) and (2.75,.5) .. node [auto] {$\scs 2$} (25); 
	\draw (3) .. controls (3.25,.5) and (3.75,.5) .. node [above=-1.5pt] {$\scs 2$} (4); 
	\draw (31) .. controls (3.5,1) and (4.5,1) .. node [auto] {$\scs 1$} (5); 
	\draw (32) .. controls (3.5,1.5) and (4.5,1.5) .. node [auto] {$\scs 1$}  (51);
\end{tikzpicture},$$
where we replace the endpoints of the arcs by their images under $\Lambda_K$.

Construct a bipartite graph $\Gamma_K(\lambda)$ given by vertices
\begin{align*}
V_\bullet &= \{j\larc{a}k\in \lambda\mid \Lambda_K(j\larc{a}k)=(\bullet,\bullet)\}\\ 
V_\circ &= \{j\larc{a}k\in \lambda\mid \Lambda_K(j\larc{a}k)=(\circ,\circ)\},
\end{align*}
and an edge from $i\larc{a}l\in V_\circ$ to $j\larc{b}k\in V_\bullet$ if $i<j<k<l$. 

In our example,
$$\Gamma_{\{1,2,3,4,5\}}(\lambda)=
\begin{tikzpicture} [baseline=0cm]
	\node  at (0,.75) {$\scs 3\slarc{2}4$};
	\node  at (1,.75) {$\scs 3\slarc{1}5$};
	\node at (0,-.75) {$\scs 2\slarc{2}3$};
	\node at  (1,-.75) {$\scs 3\slarc{1}5$};
	\node (14) at (0,.5)  [inner sep=0pt] {$\circ$};
	\node (34) at (1,.5)  [inner sep=0pt] {$\circ$};
	\node (23) at (0,-.5) [inner sep=0pt] {$\bullet$};
	\node at  (1,-.5) [inner sep=0pt] {$\bullet$};
\end{tikzpicture}
\quad\text{and}\quad 
\Gamma_{\{2,3,4,5\}}(\lambda)=
\begin{tikzpicture} [baseline=0cm]
	\node  at (0,.75) {$\scs 1\slarc{1}4$};
	\node  at (1,.75) {$\scs 3\slarc{2}4$};
	\node  at (2,.75) {$\scs 3\slarc{1}5$};
	\node at (0,-.75) {$\scs 2\slarc{2}3$};
	\node at  (1,-.75) {$\scs 3\slarc{1}5$};
	\node (14) at (0,.5)  [inner sep=0pt] {$\circ$};
	\node (34) at (1,.5)  [inner sep=0pt] {$\circ$};
	\node (25) at (2,.5)  [inner sep=0pt] {$\circ$};
	\node (23) at (0,-.5) [inner sep=0pt] {$\bullet$};
	\node at  (1,-.5) [inner sep=0pt] {$\bullet$};
	\draw (14) -- (23);
\end{tikzpicture}
.$$

The main theorem of this section follows, and its proof can be found in Section \ref{StraighteningSection}.

\begin{theorem}\label{TensorResult}
Suppose $\lambda,\mu,\nu\in \cS_K(q)$ with $K\subseteq \ZZ_{\geq 1}$ a finite subset.  Then
$$\langle \chi^\lambda\otimes \chi^\mu,\chi^{\nu}\rangle\neq 0$$
if and only if $\Gamma_K(\lambda\cup\mu\cup\bar{\nu})$ has a complete matching from $V_\bullet$ to $V_\circ$.
\end{theorem}

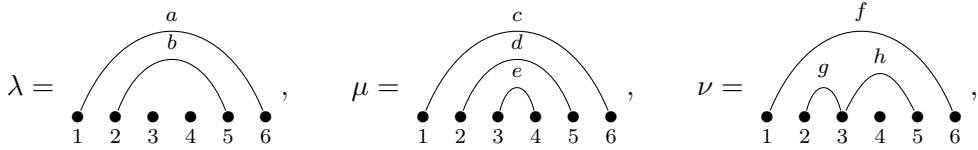
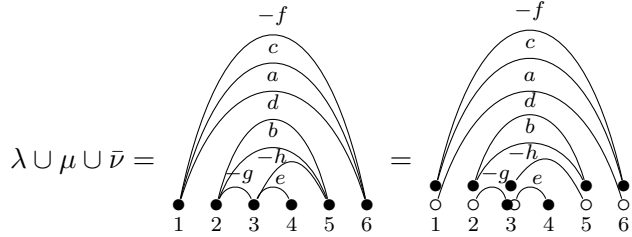
\begin{example}
Suppose 
$$\lambda=
\begin{tikzpicture}[baseline=.3cm]
	\foreach \x in {1,...,6} 
		\node (\x) at (\x/2,0) [inner sep=0pt] {$\bullet$};
	\foreach \x in {1,...,6} 
		\node at (\x/2,-.25) {$\scs\x$};
	\draw (1) .. controls (2.25/2,1.5) and (4.75/2,1.5) .. node [auto] {$\scs a$}  (6); 
	\draw (2) .. controls (2.75/2,1) and (4.25/2,1) .. node [auto] {$\scs b$}  (5); 
\end{tikzpicture},
\qquad 
\mu = 
\begin{tikzpicture}[baseline=.3cm]
	\foreach \x in {1,...,6} 
		\node (\x) at (\x/2,0) [inner sep=0pt] {$\bullet$};
	\foreach \x in {1,...,6} 
		\node at (\x/2,-.25) {$\scs\x$};
	\draw (1) .. controls (2.25/2,1.5) and (4.75/2,1.5) .. node [auto] {$\scs c$}  (6); 
	\draw (2) .. controls (2.75/2,1) and (4.25/2,1) .. node [auto] {$\scs d$}  (5); 
	\draw (3) .. controls (3.25/2,.5) and (3.75/2,.5) .. node [auto] {$\scs e$}  (4); 
\end{tikzpicture}
,\qquad 
\nu=
\begin{tikzpicture}[baseline=.3cm]
	\foreach \x in {1,...,6} 
		\node (\x) at (\x/2,0) [inner sep=0pt] {$\bullet$};
	\foreach \x in {1,...,6} 
		\node at (\x/2,-.25) {$\scs\x$};
	\draw (1) .. controls (2.25/2,1.5) and (4.75/2,1.5) .. node [auto] {$\scs f$}  (6); 
	\draw (2) .. controls (2.25/2,.5) and (2.75/2,.5) .. node [auto] {$\scs g$}  (3); 
	\draw (3) .. controls (3.75/2,.75) and (4.25/2,.75) .. node [auto] {$\scs h$}  (5); 
\end{tikzpicture},$$
so
$$\lambda\cup\mu\cup\bar\nu=\begin{tikzpicture}[baseline=.5cm]
	\foreach \x in {1,...,6} 
		\node (\x) at (\x/2,0) [inner sep=0pt] {$\bullet$};
	\foreach \x in {1,...,6} 
		\node at (\x/2,-.25) {$\scs\x$};
	\draw (1) .. controls (2.25/2,2) and (4.75/2,2) .. node [auto] {$\scs a$}  (6); 
	\draw (2) .. controls (2.75/2,1) and (4.25/2,1) .. node [auto] {$\scs b$}  (5); 
	\draw (1) .. controls (2.25/2,2.5) and (4.75/2,2.5) .. node [auto] {$\scs c$}  (6); 
	\draw (2) .. controls (2.75/2,1.5) and (4.25/2,1.5) .. node [auto] {$\scs d$}  (5); 
	\draw (3) .. controls (3.25/2,.3) and (3.75/2,.3) .. node [above=-2pt,pos=.7] {$\scs e$}  (4); 
	\draw (1) .. controls (2.25/2,3) and (4.75/2,3) .. node [auto] {$\scs -f$}  (6); 
	\draw (2) .. controls (2.25/2,.3) and (2.75/2,.3) .. node [above=-2pt,pos=.6] {$\scs -g$}  (3); 
	\draw (3) .. controls (3.75/2,.75) and (4.25/2,.75) .. node [above,pos=.2] {$\scs -h$}  (5); 
\end{tikzpicture}
=
\begin{tikzpicture}[baseline=.5cm]
	\foreach \x in {1,2,5,6} 
		\node (\x) at (\x/2,0) [inner sep=0pt] {$\circ$};
	\node (4) at (2,0) [inner sep=0pt] {$\bullet$};
	\node (3l) at (1.45,0) [inner sep=0pt] {$\bullet$};
	\node (3r) at (1.55,0) [inner sep=0pt] {$\circ$};
	\foreach \x in {1,2,3,5,6} 
		\node (1\x) at (\x/2,.25) [inner sep=0pt] {$\bullet$};
	\foreach \x in {1,...,6} 
		\node at (\x/2,-.25) {$\scs\x$};
	\draw (1) .. controls (2.25/2,2) and (4.75/2,2) .. node [auto] {$\scs a$}  (6); 
	\draw (12) .. controls (2.75/2,1) and (4.25/2,1) .. node [auto] {$\scs b$}  (15); 
	\draw (11) .. controls (2.25/2,2.5) and (4.75/2,2.5) .. node [auto] {$\scs c$}  (16); 
	\draw (12) .. controls (2.75/2,1.5) and (4.25/2,1.5) .. node [above=-1pt] {$\scs d$}  (15); 
	\draw (3r) .. controls (3.25/2,.3) and (3.75/2,.3) .. node [above=-2pt,pos=.7] {$\scs e$}  (4); 
	\draw (11) .. controls (2.25/2,3) and (4.75/2,3) .. node [auto] {$\scs -f$}  (16); 
	\draw (2) .. controls (2.25/2,.3) and (2.75/2,.3) .. node [above=-2pt,pos=.6] {$\scs -g$}  (3l); 
	\draw (13) .. controls (3.75/2,.75) and (4.25/2,.75) .. node [above,pos=.1] {$\scs -h$}  (5); 
\end{tikzpicture}
$$
where the perturbing of $\{2\larc{b}5,2\larc{d}5\}$ and $\{1\larc{c}6,1\larc{-f}6\}$ depend on $b+d$ and $a+c-f$.
Then
$$\begin{array}{|c|c|c|} \hline  \Gamma_{\{1,2,3,4,5,6\}}(\lambda\cup\mu\cup\bar\nu) & b+d=0 & b+d\neq 0\\ \hline
a+c-f=0 & \overset{1\slarc{a}6}{\circ} & 
\begin{tikzpicture} [baseline=0cm]
	\node at (0,.75) {$\scs 1\slarc{a}6$};
	\node (1a6) at (0,.5)  [inner sep=0pt] {$\circ$};
	\node (2d5) at (0,-.5)  [inner sep=0pt] {$\bullet$};
	\node  at (0,-.75) {$\scs 2\slarc{d}5$};
	\draw (1a6) -- (2d5);
\end{tikzpicture}
 \\ \hline
a+c-f\neq 0 & 
\begin{tikzpicture}[baseline=0cm]
	\node  at (0,.75) {$\scs 1\slarc{a}6$};
	\node at (1,.75) {$\scs 1\slarc{c}6$};
	\node at  (.5,-.75) {$\scs 1\slarc{-f}6$};
	\node (1a6) at (0,.5) [inner sep=0pt] {$\circ$};
	\node (1c6) at (1,.5) [inner sep=0pt] {$\circ$};
	\node (1f6) at  (.5,-.5) [inner sep=0pt] {$\bullet$};
\end{tikzpicture} & 
\begin{tikzpicture} [baseline=0cm]
	\node  at (0,.75) {$\scs 1\slarc{a}6$};
	\node  at (1,.75) {$\scs 1\slarc{c}6$};
	\node at (0,-.75) {$\scs 2\slarc{d}5$};
	\node at  (1,-.75) {$\scs 1\slarc{-f}6$};
	\node (1a6) at (0,.5)  [inner sep=0pt] {$\circ$};
	\node (1c6) at (1,.5)  [inner sep=0pt] {$\circ$};
	\node (2d5) at (0,-.5) [inner sep=0pt] {$\bullet$};
	\node (1f6) at  (1,-.5) [inner sep=0pt] {$\bullet$};
	\draw (1a6) -- (2d5);
	\draw (1c6) -- (2d5);
\end{tikzpicture}\\ \hline
\end{array}$$
Thus, the coefficient of $\chi^\nu$ is nonzero in $\chi^{\lambda\cup\mu}$ if and only if $a+c-f=0$.
\end{example}

\subsection{Straightening rules} \label{StraighteningSection}

Given $\lambda\in\cM(q)$, this section describes ``straightening" rules that allow us to create a sequence
$$\lambda=\lambda^{(0)},\lambda^{(1)}, \cdots,\lambda^{(\ell)},$$
where at each stage we remove a conflict of type (CL), (CR), or (CB) until we arrive at $\lambda^{(\ell)}\in \cS(q)$.  Furthermore, there is an underlying sequence of pairs $(K^{(0)},L^{(0)}),(K^{(1)},L^{(1)}),\ldots,(K^{(\ell)},L^{(\ell)})$ of finite subsets such that $|K^{(i)}|=|K|$ and $\lambda^{(i)}\in\cM_{K^{(i)}\cup L^{(i)}}(q)$.   While the order in which one applies the straightening rules does matter in terms of which set partition one obtains, for our purposes in this paper (Theorem \ref{StraighteningTheorem}, below) the differences are irrelevant.  The rules are as follows.

For $a,b\in \FF_q^\times$, in moving from $\lambda^{(m-1)}$ to $\lambda^{(m)}$ we can
\begin{equation}
\begin{tikzpicture}[baseline=.2cm]
		\node (1) at (0,0) [inner sep=0pt] {$\bullet$};
		\node at (0,-.25) {$\scs i$};
		\node (2) at (2,0) [inner sep=0pt] {$\bullet$};
		\node at (2,-.25) {$\scs j$};
		\node (3) at (3,0) [inner sep=0pt] {$\bullet$};
		\node at (3,-.25) {$\scs k$};		
	\draw (1) .. controls (.75,1.5) and (2.25,1.5) .. node [auto] {$\scs a$} (3); 
	\draw (1) .. controls (.5,.75) and (1.5,.75) .. node [auto] {$\scs b$} (2); 
\end{tikzpicture}
\longmapsto 
\begin{tikzpicture}[baseline=.2cm]
		\node (1) at (0,0) [inner sep=0pt] {$\bullet$};
		\node at (0,-.25) {$\scs i$};
		\node (1') at (.5,0) [inner sep=0pt] {$\circ$};
		\node at (.5,-.25) {$\scs i+1$};
		\node (2) at (2,0) [inner sep=0pt] {$\bullet$};
		\node at (2,-.25) {$\scs j+1$};
		\node (3) at (3,0) [inner sep=0pt] {$\bullet$};
		\node at (3,-.25) {$\scs k+1$};		
	\draw (1) .. controls (.75,1.5) and (2.25,1.5) ..  node [auto] {$\scs a$} (3)(3); 
	\draw (1') .. controls (1,.75) and (1.5,.75) .. node [auto] {$\scs 1$} (2); 
\end{tikzpicture}\label{StraightenLeft}\tag{SL}\end{equation}
with
\begin{align*}
 K^{(m)}  &= ([1,i]\cap K^{(m-1)})\cup ((\ZZ_{\geq i+1}\cap K^{(m-1)})+1)\\
 L^{(m)} &= ([1,i]\cap L^{(m-1)})\cup\{i+1\}\cup ((\ZZ_{\geq i+1}\cap L^{(m-1)})+1);
 \end{align*}
 \begin{equation}
\begin{tikzpicture}[baseline=.2cm]
		\node (1) at (0,0) [inner sep=0pt] {$\bullet$};
		\node at (0,-.25) {$\scs i$};
		\node (2) at (1,0) [inner sep=0pt] {$\bullet$};
		\node at (1,-.25) {$\scs j$};
		\node (3) at (3,0) [inner sep=0pt] {$\bullet$};
		\node at (3,-.25) {$\scs k$};		
	\draw (1) .. controls (.75,1.5) and (2.25,1.5) ..  node [auto] {$\scs a$} (3); 
	\draw (2) .. controls (1.5,.75) and (2.5,.75) .. node [auto] {$\scs b$} (3); 
\end{tikzpicture} 
\longmapsto 
\begin{tikzpicture}[baseline=.2cm]
		\node (1) at (0,0) [inner sep=0pt] {$\bullet$};
		\node at (0,-.25) {$\scs i$};
		\node (2) at (1,0) [inner sep=0pt] {$\bullet$};
		\node at (1,-.25) {$\scs j$};
		\node (3') at (2.5,0) [inner sep=0pt] {$\circ$};
		\node at (2.5,-.25) {$\scs k$};
		\node (3) at (3,0) [inner sep=0pt] {$\bullet$};
		\node at (3,-.25) {$\scs k+1$};		
	\draw (1) .. controls (.75,1.5) and (2.25,1.5) .. node [auto] {$\scs a$} (3); 
	\draw (2) .. controls (1.5,.75) and (2,.75) .. node [auto] {$\scs 1$} (3'); 
\end{tikzpicture} \label{StraightenRight}\tag{SR} 
\end{equation}
with
\begin{align*}
K^{(m)}&=([1,k-1]\cap K^{(m-1)})\cup ((\ZZ_{\geq k}\cap K^{(m-1)})+1),\\ 
L^{(m)}&=([1,k-1]\cap L^{(m-1)})\cup\{k\}\cup ((\ZZ_{\geq k}\cap L^{(m-1)})+1);
\end{align*}
\begin{equation}
\begin{tikzpicture}[baseline=.2cm]
		\node (1) at (0,0) [inner sep=0pt] {$\bullet$};
		\node at (0,-.25) {$\scs i$};
		\node (3) at (3,0) [inner sep=0pt] {$\bullet$};
		\node at (3,-.25) {$\scs k$};		
	\draw (1) .. controls (.75,1) and (2.25,1) .. node [auto] {$\scs a$} (3); 
	\draw (1) .. controls (.75,1.5) and (2.25,1.5) .. node [auto] {$\scs b$}(3); 
\end{tikzpicture} 
\longmapsto
\left\{\begin{array}{ll} 
\begin{tikzpicture}[baseline=.2cm]
		\node (1) at (0,0) [inner sep=0pt] {$\bullet$};
		\node at (0,-.25) {$\scs i$};
		\node (1') at (.5,0) [inner sep=0pt] {$\circ$};
		\node at (.5,-.25) {$\scs i+1$};
		\node (3') at (2.5,0) [inner sep=0pt] {$\circ$};
		\node at (2.4,-.25) {$\scs k+1$};
		\node (3) at (3,0) [inner sep=0pt] {$\bullet$};
		\node at (3.1,-.25) {$\scs k+2$};		
	\draw (1) .. controls (.75,1) and (1.75,1) .. node [auto] {$\scs 1$} (3');
	\draw (1') .. controls (1.25,1) and (2.25,1) .. node [auto] {$\scs 1$} (3); 
\end{tikzpicture}
, & \text{if $a+b=0$,}\\
\begin{tikzpicture}[baseline=.2cm]
		\node (1) at (0,0) [inner sep=0pt] {$\bullet$};
		\node at (0,-.25) {$\scs i$};
		\node (1') at (.5,0) [inner sep=0pt] {$\circ$};
		\node at (.5,-.25) {$\scs i+1$};
		\node (3') at (2.5,0) [inner sep=0pt] {$\circ$};
		\node at (2.4,-.25) {$\scs k+1$};
		\node (3) at (3,0) [inner sep=0pt] {$\bullet$};
		\node at (3.1,-.25) {$\scs k+2$};		
	\draw (1) .. controls (.75,1.25) and (2.25,1.25) .. node [auto] {$\scs a+b$} (3);
	\draw (1') .. controls (1,.75) and (2,.75) .. node [auto] {$\scs 1$} (3'); 
\end{tikzpicture},
& \text{if $a+b\neq 0$.}\end{array}\right. \label{StraightenDouble}\tag{SB}
\end{equation}
with
\begin{align*}
K^{(m)}&=([1,i]\cap K^{(m-1)})\cup (([i+1,k-1]\cap K^{(m-1)})+1)\cup((\ZZ_{\geq k}\cap K^{(m-1)})+2)\\
L^{(m)}&=([1,i]\cap L^{(m-1)})\cup\{i+1\} (([i+1,k-1]\cap L^{(m-1)})+1)\cup\{k+1\}\cup((\ZZ_{\geq k}\cap L^{(m-1)})+2).
\end{align*}

In each case there are  ``new nodes" indicated by $\circ$ that push all the other node values to the right up (note, we view the $\bullet$-nodes as being the same though their number labels may change).  In fact, $K^{(m)}$ is the set of original nodes (up to being pushed around) and  $L^{(m)}$ is the set of nodes that were at some point $\circ$-nodes (see example after Theorem \ref{StraighteningTheorem}). 

The following lemma states that these rules (SL), (SR) and (SB) fundamentally do not change the underlying character.

\begin{lemma} \label{OneStraightening}
Let $\lambda\in \cM_K(q)$ and apply (SL), (SR) or (SB) to obtain $\tilde\lambda=\lambda^{(1)}\in \cM_{K^{(1)}\cup L^{(1)}}(q)$.  Then as a character of $U_K\cong U_{K^{(1)}}$,
$$ \chi^\lambda=q^{-r_{K^{(1)}}^{K^{(1)}\cup L^{(1)}}(\tilde\lambda)}\Res_{U_{K^{(1)}}}^{U_{K^{(1)}\cup L^{(1)}}}(\chi^{\tilde\lambda}).$$
\end{lemma}
\begin{proof}
Let $(i\larc{a}l,j\larc{b}k)$ be the conflict in $\lambda$ that gets resolved to obtain $\tilde\lambda$.  If $i=j<k<l$, then we applied (SL) to get $\tilde\lambda$ with $K'=K^{(1)}=[1,i]\cap K\cup ((\ZZ_{\geq i+1}\cap K)+1)$ and $L'=L^{(1)}=\{i+1\}$.  On the one hand,  by Theorem \ref{ArcTensorProduct},
\begin{align*}
\chi^\lambda & = \chi^{\lambda-\{i\slarc{a}l,i\slarc{b}k\}}\otimes \chi^{\{i\slarc{a}l,j\slarc{b}k\}}\\
&=\chi^{\lambda-\{i\slarc{a}l,i\slarc{b}k\}} \otimes \bigg(\chi^{i\slarc{a}l}+\sum_{i<j'<k\atop j'\in K,b'\in \FF_q^\times} \chi^{\{i\slarc{a}l,j'\slarc{b'}k\}}\bigg)\\
&=\chi^{\lambda-\{i\slarc{b}k\}} \otimes \bigg(\chi^{\emptyset}+\sum_{i<j'<k\atop j'\in K,b'\in \FF_q^\times} \chi^{\{j'\slarc{b'}k\}}\bigg).
\end{align*}
On the other hand, by Theorem \ref{ArcRestriction},
\begin{align*}
q^{-r_{K'}^{K'\cup L'}(\tilde\lambda)}\Res_{U_{K'}}^{U_{K'\cup L'}}(\chi^{\tilde\lambda}) &= q^{-r_{K'}^{K'\cup L'}(\tilde\lambda)} \Res_{U_{K'}}^{U_{K'\cup L'}}(\chi^{\tilde\lambda-\{i+1\slarc{1}k+1\}})\otimes \Res_{U_{K'}}^{U_{K'\cup L'}}(\chi^{\{i+1\slarc{1}k+1\}})\\ 
&= \chi^{\tilde\lambda-\{i+1\slarc{1}k+1\}}\otimes \bigg(\chi^{\emptyset}+\sum_{i+1<j'<k+1\atop j'\in K',b'\in \FF_q^\times} \chi^{\{j'\slarc{b'}k+1\}}\bigg).
\end{align*}
These are the same characters up to re-ordering the indices between $K$ and $K'$.  The proofs for (SR) and (SB) are analogous.
\end{proof}

By iterating Lemma \ref{OneStraightening} to remove all the conflicts of a multiset, we see that up to shifting of indices every tensor product is the same  (up to a scalar multiple) as restriction from some supercharacter. 

\begin{theorem} \label{StraighteningTheorem}
Let $\lambda\in \cM_K(q)$.  Then there exists $\tilde{\lambda}\in \cS_{K'\cup L'}(q)$ with $|K|=|K'|$, such that   
$$\chi^\lambda=q^{-r_{K'}^{K'\cup L'}(\tilde\lambda)}\Res_{U_{K'}}^{U_{K'\cup L'}}(\chi^{\tilde\lambda}).$$
\end{theorem}

\begin{example}
Suppose $a,b,c,d,e,f\in \FF_q^\times$ with  $a+b+c=0$, and
$$\lambda=
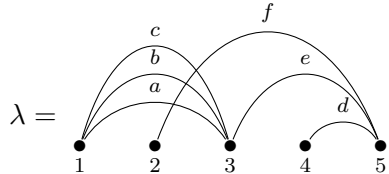
\begin{tikzpicture}[baseline=.3cm]
	\foreach \x in {1,...,5} 
		\node (\x) at (\x,0) [inner sep=0pt] {$\bullet$};
	\foreach \x in {1,...,5} 
		\node at (\x,-.25) {$\scs\x$};
	\draw (1) .. controls (1.5,1.25) and (2.5,1.25) .. node [auto] {$\scs b$}  (3); 
	\draw (1) .. controls (1.5,.75) and (2.5,.75) .. node [auto] {$\scs a$}  (3); 
	\draw (1) .. controls (1.5,1.75) and (2.5,1.75) .. node [auto] {$\scs c$}  (3); 
	\draw (2) .. controls (2.75,2) and (4.25,2) .. node [auto] {$\scs f$}  (5); 
	\draw (4) .. controls (4.25,.4) and (4.75,.4)  .. node [auto] {$\scs d$}  (5);
	\draw (3) .. controls (3.5,1.25) and (4.5,1.25)  .. node [auto] {$\scs e$}   (5);
\end{tikzpicture}$$
Then we can resolve conflicts as follows
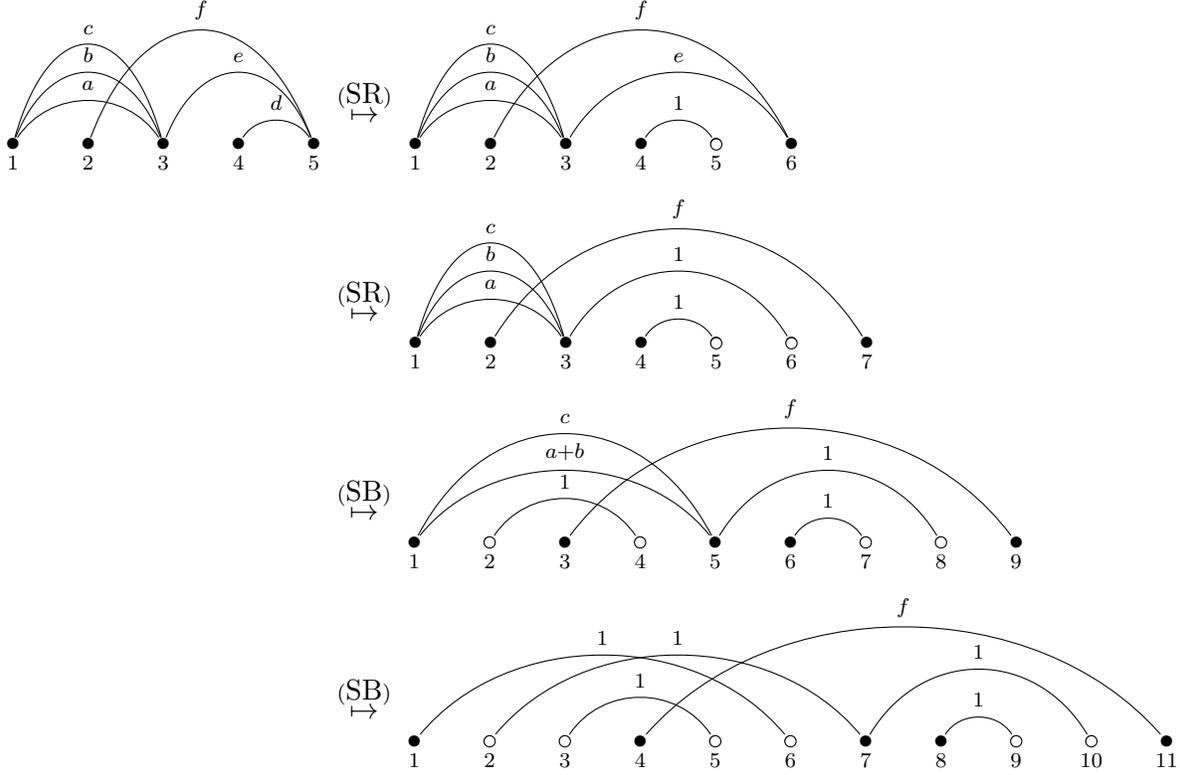
\begin{align*}
\begin{tikzpicture}[baseline=.3cm]
	\foreach \x in {1,...,5} 
		\node (\x) at (\x,0) [inner sep=0pt] {$\bullet$};
	\foreach \x in {1,...,5} 
		\node at (\x,-.25) {$\scs\x$};
	\draw (1) .. controls (1.5,1.25) and (2.5,1.25) .. node [auto] {$\scs b$}  (3); 
	\draw (1) .. controls (1.5,.75) and (2.5,.75) .. node [auto] {$\scs a$}  (3); 
	\draw (1) .. controls (1.5,1.75) and (2.5,1.75) .. node [auto] {$\scs c$}  (3); 
	\draw (2) .. controls (2.75,2) and (4.25,2) .. node [auto] {$\scs f$}  (5); 
	\draw (4) .. controls (4.25,.4) and (4.75,.4)  .. node [auto] {$\scs d$}  (5);
	\draw (3) .. controls (3.5,1.25) and (4.5,1.25)  .. node [auto] {$\scs e$}   (5);
\end{tikzpicture}
&\overset{(\ref{StraightenRight})}{\mapsto}
\begin{tikzpicture}[baseline=.3cm]
	\foreach \x in {1,...,6} 
		\node (\x) at (\x,0) [inner sep=0pt] {$\bullet$};
	\foreach \x in {1,...,6} 
		\node at (\x,-.25) {$\scs\x$};
	\draw (1) .. controls (1.5,1.25) and (2.5,1.25) .. node [auto] {$\scs b$}  (3); 
	\draw (1) .. controls (1.5,.75) and (2.5,.75) .. node [auto] {$\scs a$}  (3); 
	\draw (1) .. controls (1.5,1.75) and (2.5,1.75) .. node [auto] {$\scs c$}  (3); 
	\draw (2) .. controls (3,2) and (5,2) .. node [auto] {$\scs f$}  (6); 
	\draw (4) .. controls (4.25,.4) and (4.75,.4)  .. node [auto] {$\scs 1$}  (5);
	\draw (3) .. controls (3.75,1.25) and (5.25,1.25)  .. node [auto] {$\scs e$}   (6);
	\node at (5) [shape=circle,fill=white,inner sep=1.5pt, draw] {};
\end{tikzpicture}\\
&\overset{(\ref{StraightenRight})}{\mapsto}
\begin{tikzpicture}[baseline=.3cm]
	\foreach \x in {1,...,7} 
		\node (\x) at (\x,0) [inner sep=0pt] {$\bullet$};
	\foreach \x in {1,...,7} 
		\node at (\x,-.25) {$\scs\x$};
	\draw (1) .. controls (1.5,1.25) and (2.5,1.25) .. node [auto] {$\scs b$}  (3); 
	\draw (1) .. controls (1.5,.75) and (2.5,.75) .. node [auto] {$\scs a$}  (3); 
	\draw (1) .. controls (1.5,1.75) and (2.5,1.75) .. node [auto] {$\scs c$}  (3); 
	\draw (2) .. controls (3.25,2) and (5.75,2) .. node [auto] {$\scs f$}  (7); 
	\draw (4) .. controls (4.25,.4) and (4.75,.4)  .. node [auto] {$\scs 1$}  (5);
	\draw (3) .. controls (3.75,1.25) and (5.25,1.25)  .. node [auto] {$\scs 1$}   (6);
	\node at (5) [shape=circle,fill=white,inner sep=1.5pt, draw] {};
	\node at (6) [shape=circle,fill=white,inner sep=1.5pt, draw] {};
\end{tikzpicture}\\
&\overset{(\ref{StraightenDouble})}{\mapsto}
\begin{tikzpicture}[baseline=.3cm]
	\foreach \x in {1,...,9} 
		\node (\x) at (\x,0) [inner sep=0pt] {$\bullet$};
	\foreach \x in {1,...,9} 
		\node at (\x,-.25) {$\scs\x$};
	\draw (1) .. controls (2,1.25) and (4,1.25) .. node [auto] {$\scs a+b$}  (5); 
	\draw (2) .. controls (2.5,.75) and (3.5,.75) .. node [auto] {$\scs 1$}  (4); 
	\draw (1) .. controls (2,1.9) and (4,1.9) .. node [auto] {$\scs c$}  (5); 
	\draw (3) .. controls (4.5,2) and (7.5,2) .. node [auto] {$\scs f$}  (9); 
	\draw (6) .. controls (6.25,.4) and (6.75,.4)  .. node [auto] {$\scs 1$}  (7);
	\draw (5) .. controls (5.75,1.25) and (7.25,1.25)  .. node [auto] {$\scs 1$}   (8);
	\node at (2) [shape=circle,fill=white,inner sep=1.5pt, draw] {};
	\node at (4) [shape=circle,fill=white,inner sep=1.5pt,draw] {};
	\node at (7) [shape=circle,fill=white,inner sep=1.5pt, draw] {};
	\node at (8) [shape=circle,fill=white,inner sep=1.5pt, draw] {};
\end{tikzpicture}\\
&\overset{(\ref{StraightenDouble})}{\mapsto}
\begin{tikzpicture}[baseline=.3cm]
	\foreach \x in {1,...,11} 
		\node (\x) at (\x,0) [inner sep=0pt] {$\bullet$};
	\foreach \x in {1,...,11} 
		\node at (\x,-.25) {$\scs\x$};
	\draw (1) .. controls (2.25,1.5) and (4.75,1.5) .. node [auto] {$\scs 1$}  (6); 
	\draw (3) .. controls (3.5,.75) and (4.5,.75) .. node [auto] {$\scs 1$}  (5); 
	\draw (2) .. controls (3.25,1.5) and (5.75,1.5) .. node [auto] {$\scs 1$}  (7); 
	\draw (4) .. controls (5.75,2) and (9.25,2) .. node [auto] {$\scs f$}  (11); 
	\draw (8) .. controls (8.25,.4) and (8.75,.4)  .. node [auto] {$\scs 1$}  (9);
	\draw (7) .. controls (7.75,1.25) and (9.25,1.25)  .. node [auto] {$\scs 1$}   (10);
	\node at (2) [shape=circle,fill=white,inner sep=1.5pt, draw] {};
	\node at (3) [shape=circle,fill=white,inner sep=1.5pt, draw] {};
	\node at (5) [shape=circle,fill=white,inner sep=1.5pt, draw] {};
	\node at (6) [shape=circle,fill=white,inner sep=1.5pt, draw] {};
	\node at (9) [shape=circle,fill=white,inner sep=1.5pt, draw] {};
	\node at (10) [shape=circle,fill=white,inner sep=1.5pt, draw] {};
\end{tikzpicture}
\end{align*}
where the $\bullet$-nodes are in $K^{(i)}$ and the $\circ$-nodes are in $L^{(i)}$.  For example, $K^{(4)}=\{1,4,7,8,11\}$ and $L^{(4)}=\{2,3,5,6,9,10\}$.  Note that the order in which we resolved the conflicts above minimized the number of crossings in the resulting set partition.  On the other hand, by picking an alternate order of straightening rules, we might have ended up with the $q$-set partition
$$ 
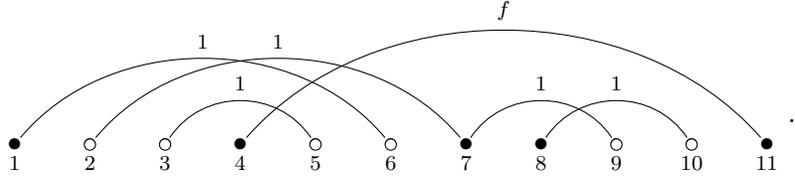
\begin{tikzpicture}[baseline=.3cm]
	\foreach \x in {1,...,11} 
		\node (\x) at (\x,0) [inner sep=0pt] {$\bullet$};
	\foreach \x in {1,...,11} 
		\node at (\x,-.25) {$\scs\x$};
	\draw (1) .. controls (2.25,1.5) and (4.75,1.5) .. node [auto] {$\scs 1$}  (6); 
	\draw (3) .. controls (3.5,.75) and (4.5,.75) .. node [auto] {$\scs 1$}  (5); 
	\draw (2) .. controls (3.25,1.5) and (5.75,1.5) .. node [auto] {$\scs 1$}  (7); 
	\draw (4) .. controls (5.75,2) and (9.25,2) .. node [auto] {$\scs f$}  (11); 
	\draw (8) .. controls (8.5,.75) and (9.5,.75)  .. node [auto] {$\scs 1$}  (10);
	\draw (7) .. controls (7.5,.75) and (8.5,.75)  .. node [auto] {$\scs 1$}   (9);
	\node at (2) [shape=circle,fill=white,inner sep=1.5pt, draw] {};
	\node at (3) [shape=circle,fill=white,inner sep=1.5pt, draw] {};
	\node at (5) [shape=circle,fill=white,inner sep=1.5pt, draw] {};
	\node at (6) [shape=circle,fill=white,inner sep=1.5pt, draw] {};
	\node at (9) [shape=circle,fill=white,inner sep=1.5pt, draw] {};
	\node at (10) [shape=circle,fill=white,inner sep=1.5pt, draw] {};
\end{tikzpicture} .$$
However, this other set partition restricts to the same character of $U_{K^{(4)}}(q)$.
\end{example}

\begin{remark}
Note that if  we assume there are no (CN) conflicts, or $\lambda\in \cM_K(q)$, then we may simplify the definition of the labeling function $\Lambda_K$ as follows.  If $\lambda\in \cM_K(q)$ and $\tilde\lambda\in \cS_{K^{(\ell)}\cup L^{(\ell)}}(q)$ is obtained by applying (SL), (SR) and (SB), then 
$$\Lambda_K(j\larc{b}k)=(\Lambda_K^L(j\larc{b}k),\Lambda_K^R(j\larc{b}k)),$$
where 
\begin{align*}
\Lambda_K^L(j\larc{b}k)&=\left\{\begin{array}{ll} \circ & \text{if $j\in L^{(\ell)}$,}\\ \bullet, &\text{otherwise.}\end{array}\right.\\
\Lambda_K^R(j\larc{b}k)&=\left\{\begin{array}{ll} \circ & \text{if $k\in L^{(\ell)}$,}\\ \bullet, &\text{otherwise.}\end{array}\right.
\end{align*}

\end{remark}

We now have sufficient tools to prove Theorem  \ref{TensorResult}.

\begin{proof}[Proof of Theorem \ref{TensorResult}] 
First note that 
$$\langle \chi^\lambda\otimes \chi^\mu,\chi^\nu\rangle_{U_K}=\langle \chi^\lambda\otimes \chi^\mu\otimes\chi^{\bar\nu},\One\rangle_{U_K}.$$
 By Theorem \ref{StraighteningTheorem} there exists $\widetilde{\lambda\cup\mu\cup\bar{\nu}}\in\cS_{K'\cup L'}(q)$ such that 
$$\chi^{\lambda\cup\mu\cup\bar{\nu}}=q^{-r_{K'}^{K'\cup L'}(\widetilde{\lambda\cup\mu\cup\bar{\nu}})}\Res_{U_{K'}}^{U_{K'\cup L'}}(\chi^{\widetilde{\lambda\cup\mu\cup\bar{\nu}}}).$$
Thus,
$$\langle \chi^{\lambda\cup\mu\cup\bar\nu},\One\rangle_{U_K}= q^{-r_{K'}^{K'\cup L'}(\widetilde{\lambda\cup\mu\cup\bar{\nu}})} \langle \Res_{U_{K'}}^{U_{K'\cup L'}}(\chi^{\widetilde{\lambda\cup\mu\cup\bar{\nu}}}),\One\rangle_{U_{K'}},$$
and the result follows from Theorem \ref{MainTheorem} and the remark preceding this proof.
\end{proof}

\subsection{Consequences for restriction coefficients}

Theorem \ref{StraighteningTheorem} also allows us to extend Theorem \ref{MainTheorem} to the coefficient of arbitrary supercharacters.

\begin{theorem}\label{RestrictionResult}
Suppose $\lambda\in \cS_L(q)$ and $\mu\in \cS_K(q)$ with $K\subseteq L\subseteq \ZZ_{\geq 1}$ finite sets.  Then
$$\langle \Res_{U_K}^{U_L}(\chi^\lambda), \chi^\mu\rangle\neq 0$$
if and only if $\Gamma_K(\lambda\cup\bar\mu)$ has a complete matching from $V_\bullet$ to $V_\circ$.
\end{theorem}

\begin{proof}  Note that 
$$\langle \Res_{U_K}^{U_L}(\chi^\lambda), \chi^\mu\rangle=\langle \Res_{U_K}^{U_L}(\chi^\lambda)\otimes \chi^{\bar\mu},\One\rangle.$$
Since $\mu\in\cS_K(q)$, by (\ref{AlreadyDownArcsRestriction})
$$\Res_{U_K}^{U_L}(\chi^\lambda)\otimes \chi^{\bar\mu}=q^{-r_K^L(\bar\mu)} \Res_{U_K}^{U_L}(\chi^\lambda\otimes \chi^{\bar\mu})=q^{-r_K^L(\bar\mu)} \Res_{U_K}^{U_L}(\chi^{\lambda\cup\bar\mu}).$$
By Theorem \ref{StraighteningTheorem} there exist, $L'\subseteq M\subseteq \ZZ_{\geq 1}$, and $\widetilde{\lambda\cup\bar\mu}\in \cS_{M}(q)$ such that as characters of $U_L(q)\cong U_{L'}(q)$,
$$\chi^{\lambda\cup\bar{\mu}}=q^{-r_{L'}^{M}(\widetilde{\lambda\cup\bar{\mu}})}\Res_{U_{L'}}^{U_M}(\chi^{\widetilde{\lambda\cup\bar\mu}}).$$
Thus, there is $K'\subseteq L'$ such that
$$\langle \Res_{U_K}^{U_L}(\chi^\lambda), \chi^\mu\rangle=q^{- r_{L'}^{M}(\widetilde{\lambda\cup\bar{\mu}})-r_K^L(\bar\mu)}\langle \Res_{U_{K'}}^{U_M}(\chi^{\widetilde{\lambda\cup\bar\mu}}), \One\rangle.$$
The result now follows from Theorem \ref{MainTheorem}.
\end{proof}

\section{Explicit coefficients}

This section explicitly computes some of the coefficients of the trivial character in the restriction and the tensor products for a family of examples.

For $\lambda\in \cS_L(q)$ and $K\subseteq L$, let 
$$\cC_K(\lambda)=\{ (i\larc{a}k,j\larc{b}l)\in \lambda\mid i<j<k<l, i,l\notin K, j,k\in K\}.$$
A crossing $(i\larc{a}k,j\larc{b}l)$ is \emph{maximal} in a set $B$ of crossings if $(i'\larc{a}k',j\larc{b}l)\in B$ implies $k'<k$ and $(i\larc{a}k,j'\larc{b}l')\in B$ implies $j<j'$.  

The main theorem of this section gives an explicit computation of $\langle \Res_{U_K}^{U_L}(\chi^\lambda),\One\rangle$ in the case when $\lambda$ has a particular convenient labeling in the sense of (\ref{LabellingRules}).

\begin{theorem} \label{ExplicitCoefficientsTheorem}  Let $K\subseteq L\subseteq \ZZ_{\geq 1}$ be finite subsets.
Suppose $\lambda\in\cS_L(q)$ satisfies $j\larc{a}k\in\lambda$  only if $|\{j,k\}\cap K|=1$.  Then
$$\langle \Res^{U_L}_{U_K} (\chi^\lambda),\One\rangle=q^{r_K^L(\lambda)}q^{|\cC_K(\lambda)|}.$$
\end{theorem}

\begin{example} If  $K=\{4,5,7,8,9,10\}$ and 
$$\lambda=
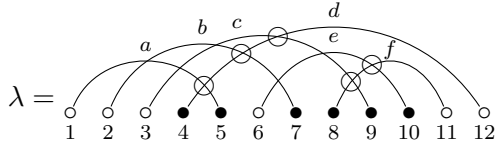
\begin{tikzpicture}[baseline = .1cm]
    \foreach \y in {1,2,3,6,11,12}
        \node(\y) at (\y/2,0) [inner sep=0pt] {$\circ$};
    \foreach \z in {4,5,7,8,9,10}
        \node(\z) at (\z/2,0) [inner sep=0pt] {$\bullet$};
    \foreach \x in {1,...,12}
        \node at (\x/2,-.25) [inner sep=0pt] {$\scs\x$};
    \draw (1) .. controls (1.75/2, .6*1.5) and (4.25/2, .6*1.5) .. node [above] {$\scs a $} (5);
    \draw (2) .. controls (3/2, .8*1.5) and (6/2, .8*1.5) .. node [above] {$\scs b $}  (7);
    \draw (3) .. controls (4.5/2, .9*1.5) and (7.75/2, .9*1.5) .. node [above,pos=.4] {$\scs c $}  (9);
    \draw (4) .. controls (6/2, 1*1.5) and (10/2, 1*1.5) .. node [above] {$\scs d $}  node [pos=.21] {$\scs \bigcirc$} node [pos=.07] {$\scs \bigcirc$} node [pos=.33] {$\scs \bigcirc$} (12);
    \draw (6) .. controls (7/2, .7*1.5) and (9/2, .7*1.5) .. node [above] {$\scs e $}  (10) ;
    \draw (8) .. controls (8.75/2, .6*1.5) and (10.25/2, .6*1.5) .. node [above=-3pt] {$\scs f $}  node [pos=.16] {$\scs \bigcirc$} node [pos=.35] {$\scs\bigcirc$} (11);
\end{tikzpicture}$$
where the elements of $\cC_K(q)$ are circled, then Theorem \ref{ExplicitCoefficientsTheorem} says
$$\langle \Res^{U_L}_{U_K} (\chi^\lambda),\One\rangle= q^{7} q^{5}.$$
\end{example}

To prove Theorem \ref{ExplicitCoefficientsTheorem} we note that 
$$\Res^{U_L}_{U_K} (\chi^\lambda)=\bigotimes_{i\slarc{a}l\in \lambda} \Res^{U_L}_{U_K} (\chi^{i\slarc{a}l})$$
so by our assumption that $j\larc{a}k\in\lambda$ only if $|\{j,k\}\cap K|=1$, we can write
\begin{align}
\Res^{U_L}_{U_K} (\chi^\lambda)& =q^{r_K^L(\lambda)}\bigotimes_{i\slarc{a}l\in \lambda\atop i\notin K} \bigg(\One+ \sum_{i<j<l\atop j\in K, b\in\FF_q^\times} \chi^{j\slarc{b}l}\bigg)\otimes \bigotimes_{i\slarc{a}l\in \lambda\atop l\notin K} \bigg(\One+\sum_{i<k<l\atop k\in K, b\in\FF_q^\times} \chi^{i\slarc{b}k}\bigg) \label{FirstIterationProduct}\\
&=q^{r_K^L(\lambda)}\sum_{\gamma}c_\gamma \chi^\gamma, \label{FirstIterationSum}
\end{align}
where $c_\gamma\in \{0,1\}$, and each $\gamma\in\cM_K(q)$ with $c_\gamma=1$ comes from one choice of terms in the product (\ref{FirstIterationProduct}).  In other words, if
$$\mathcal{I}_{i\slarc{a}l}(\lambda)=\left\{\begin{array}{ll} \{\emptyset,j\larc{b}l\mid j\in K\}, & \text{if $i\notin K$,}\\ \{\emptyset,i\larc{b}k\mid k\in K\}, & \text{if $l\notin K$,}\end{array}\right.$$
then there is an ordering of the arcs of $\gamma$ such that as a sequence
$$\gamma\in \cI_K(\lambda),$$
where $\cI_K(\lambda)$ is the Cartesian product of the sets $\cI_{i\slarc{a}l}(\lambda)$ for all $i\larc{a}l\in\lambda$ in some fixed order.
  If we write $\gamma\in \cI_K(\lambda)$, we assume we have fixed some ordering on $\gamma$ and we will take $j\larc{b}k\in \gamma\cap \cI_{i\slarc{a}l}(\lambda)$ to mean that the $i\larc{a}l$th arc of $\gamma$ is $j\larc{b}k$.  Note that any such $\gamma\in \cI_K(\lambda)$ will satisfy $0\leq m_{jk}(\gamma)\leq 2$ for all $j<k\in K$.   For example, if 
  $$\lambda=
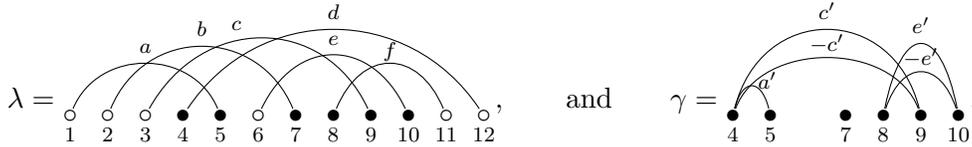
\begin{tikzpicture}[baseline = .1cm]
    \foreach \y in {1,2,3,6,11,12}
        \node(\y) at (\y/2,0) [inner sep=0pt] {$\circ$};
    \foreach \z in {4,5,7,8,9,10}
        \node(\z) at (\z/2,0) [inner sep=0pt] {$\bullet$};
    \foreach \x in {1,...,12}
        \node at (\x/2,-.25) [inner sep=0pt] {$\scs\x$};
    \draw (1) .. controls (1.75/2, .6*1.5) and (4.25/2, .6*1.5) .. node [above] {$\scs a $} (5);
    \draw (2) .. controls (3/2, .8*1.5) and (6/2, .8*1.5) .. node [above] {$\scs b $}  (7);
    \draw (3) .. controls (4.5/2, .9*1.5) and (7.75/2, .9*1.5) .. node [above,pos=.4] {$\scs c $}  (9);
    \draw (4) .. controls (6/2, 1*1.5) and (10/2, 1*1.5) .. node [above] {$\scs d $}   (12);
    \draw (6) .. controls (7/2, .7*1.5) and (9/2, .7*1.5) .. node [above] {$\scs e $}  (10) ;
    \draw (8) .. controls (8.75/2, .6*1.5) and (10.25/2, .6*1.5) .. node [above=-3pt] {$\scs f $}   (11);
\end{tikzpicture}, 
\qquad \text{and}\qquad \gamma=\begin{tikzpicture}[baseline = .1cm]
    \foreach \z in {4,5,7,8,9,10}
        \node(\z) at (\z/2,0) [inner sep=0pt] {$\bullet$};
    \foreach \x in {4,5,7,8,9,10}
        \node at (\x/2,-.25) [inner sep=0pt] {$\scs\x$};
    \draw (4) .. controls (4.25/2, .5) and (4.75/2, .5) .. node [above=2pt,pos=.95] {$\scs a' $} (5);
    \draw (4) .. controls (4.75/2, 1) and (8.25/2, 1) .. node [above=-2pt] {$\scs -c' $}  (9);
    \draw (4) .. controls (4.75/2, 1.5) and (8.25/2, 1.5) .. node [above] {$\scs c' $}   (9);
    \draw (8) .. controls (8.5/2, .75) and (9.5/2, .75) .. node [above=-2pt] {$\scs -e' $} (10) ;
    \draw (8) .. controls (8.5/2, 1.25) and (9.5/2, 1.25) .. node [above] {$\scs e' $}   (10);
\end{tikzpicture},$$
then
\begin{align*}
(4\larc{a'} 5,\emptyset, &4\larc{c'}9,4\larc{-c'}9,8\larc{e'}10,8\larc{-e'} 10)\\
&\in \cI_{\{4,5,7,8,9,10\}}(\lambda)=\cI_{1\slarc{a}5}(\lambda)\times \cI_{2\slarc{b}7}(\lambda)\times \cI_{3\slarc{c}9}(\lambda)\times \cI_{4\slarc{d}12}(\lambda)\times \cI_{6\slarc{e}10}(\lambda)\times \cI_{8\slarc{f}11}(\lambda),
\end{align*}
and $4\larc{c'}9\in  \cI_{3\slarc{c}9}(\lambda)$.

The first lemma examines the structure of those $\gamma\in \cM_K(q)$ that satisfy $\Lambda_K(\gamma)\subseteq \{(\bullet,\circ),(\circ,\bullet)\}$.

\begin{lemma}\label{HalfInLemma} Let $\gamma\in\cM_K(q)$.  Then $\Lambda_K(\gamma)\subseteq \{(\circ,\bullet),(\bullet,\circ)\}$ if and only if for all $j\larc{a}k\in \gamma$, either
\begin{enumerate}
\item[(1)] $m_{jk}(\gamma)=2$, $\wt_{jk}(\gamma)=0$, $i\larc{b}k\in \gamma$ implies $j\leq i$, and $j\larc{b}l\in \gamma$ implies $l\leq k$,
\item[(2)] $m_{jk}(\gamma)=1$ and there exists either $j\larc{b}l\in \gamma$ with $l>kl$ or $i\larc{b}k\in \gamma$ with $i<j$, but not both.
\end{enumerate}
\end{lemma}
\begin{proof}
Suppose $\Lambda_K(\gamma)\subseteq \{(\circ,\bullet),(\bullet,\circ)\}$. Let $j\larc{a}k\in\gamma$. By (TB) and (\ref{LabellingRules}), a large multiplicity $m_{jk}(\gamma)>2$ implies $(\circ,\circ)\in \Lambda_K(\gamma)$.   If $m_{jk}(\gamma)=2$, then by (TB) and (\ref{LabellingRules}) $\wt_{jk}(\lambda)=0$,  and there cannot be  any endpoint above the endpoints of $j\larc{a}k$ in the perturbed version of $\gamma$.  If $\wt_{jk}(\gamma)=1$, then exactly one of the endpoints must have an endpoint above it in the perturbed diagram.  

The converse is immediate.
\end{proof}

The following will allow us to set up a recursion, since the $\gamma\in \cI_K(\lambda)$ of interest will satisfy $\Lambda_K(\gamma)\subseteq \{(\bullet,\circ),(\circ,\bullet)\}$.

\begin{lemma} \label{NonzeroOneLemma} Suppose $\lambda\in\cS_L(q)$ satisfies $j\larc{a}k\in\lambda$ only if $|\{j,k\}\cap K|=1$, and let $\gamma\in \cI_K(\lambda)$.  Then $\langle \chi^\gamma,\One\rangle\neq 0$ if and only if $\Lambda_K(\gamma)\subseteq \{(\circ,\bullet),(\bullet,\circ)\}.$
\end{lemma}
\begin{proof}
Suppose that $\langle \chi^\gamma,\One\rangle\neq 0$.   If $\Lambda_K(\gamma)\nsubseteq \{(\circ,\bullet),(\bullet,\circ)\}$, then Theorem \ref{TensorResult}  implies $(\circ,\circ)\in \Lambda_K(\gamma)$ (since for every $(\bullet,\bullet)$ there must at least one $(\circ,\circ)$).  In this case, since  $0\leq m_{jk}(\gamma)\leq 2$ the arc mapping to $(\circ,\circ)$ must be in one of the two cases,
$$
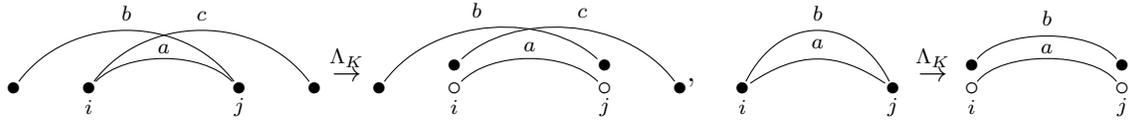
\begin{tikzpicture}[baseline=.1cm]
	\node at (2,-.25) {$\scs i$};
	\node at (4,-.25) {$\scs j$};
	\foreach \y in {1,2,4,5}
		\node (\y) at (\y,0) [inner sep=0pt] {$\bullet$};
	\draw (1) .. controls (1.75,1) and (3.25,1) ..  node [auto] {$\scs b$} (4);
    \draw (2) .. controls (2.75,1) and (4.25,1) .. node [auto] {$\scs c$} (5);
    \draw (2) .. controls (2.5,.5) and (3.5,.5) .. node [above=-2pt] {$\scs a$} (4);
\end{tikzpicture}
\overset{\Lambda_K}\rightarrow
\begin{tikzpicture}[baseline=.1cm]
	\node at (2,-.25) {$\scs i$};
	\node at (4,-.25) {$\scs j$};
	\foreach \y in {1,5}
		\node (\y) at (\y,0) [inner sep=0pt] {$\bullet$};
    \foreach \z in {2,4}
        \node (\z) at (\z, 0) [inner sep = 0pt]{$\circ$};
    \foreach \x in {2,4}
        \node (B\x) at (\x, .3) [inner sep = 0pt]{$\bullet$};
	\draw (1) .. controls (1.75,1) and (3.25,1) ..  node [auto] {$\scs b$}  (B4);
    \draw (B2) .. controls (2.75,1) and (4.25,1) .. node [auto] {$\scs c$} (5);
    \draw (2) .. controls (2.5,.5) and (3.5,.5) .. node  [above=-1pt] {$\scs a$} (4);
\end{tikzpicture},\quad
\begin{tikzpicture}[baseline=.1cm]
	\node at (1,-.25) {$\scs i$};
	\node at (3,-.25) {$\scs j$};
    \foreach\x in {1,3}
        \node(\x) at (\x,0) [inner sep=0pt] {$\bullet$};
    \draw (1) .. controls (1.5,1) and (2.5,1) .. node [auto] {$\scs b$} (3);
    \draw (1) .. controls (1.75,.5) and (2.25,.5) .. node [above] {$\scs a$}  (3);
\end{tikzpicture}
\overset{\Lambda_K}\rightarrow
\begin{tikzpicture}[baseline = .1cm]
	\node at (1,-.25) {$\scs i$};
	\node at (3,-.25) {$\scs j$};
    \foreach\x in {1,3}
        \node(\x) at (\x,0) [inner sep=0pt] {$\circ$};
    \foreach\x in {1,3}
        \node(B\x) at (\x,.3) [inner sep = 0pt] {$\bullet$};
    \draw (1) .. controls (1.4,.5) and (2.6,.5) .. node [above=-2pt] {$\scs a$}  (3);
    \draw (B1) .. controls (1.4,.8) and (2.6,.8) .. node [auto] {$\scs b$} (B3);
\end{tikzpicture}$$
where $a+b\neq 0$.

Suppose $\beta=j\larc{a}k$  satisfies $\Lambda_K(\beta)=(\circ,\circ)$ and 
\begin{equation}\label{MaximalityCondition}
\text{if $\Lambda_K(i\larc{b}l)=(\circ,\circ)$ with $i\leq j<k\leq l$, then $i\larc{b}l=\beta$.}
\end{equation}
  WLOG assume that $\beta \in \cI_{i\slarc{a}k}(\lambda)$ with  $\Lambda_K(i\larc{a}k)=(\circ,\bullet)$.  Since $\lambda$ is a $q$-set partition, the other arc $\alpha$ in $\gamma$ ending at $k$ must satisfy $\alpha\in \cI_{i'\slarc{a'}l'}(\lambda)$ with $\Lambda_K(i'\larc{a'}l')=(\bullet,\circ)$.  If $\Lambda_K(\alpha)=(\bullet,\bullet)$ we arrive at a contradiction since by (\ref{MaximalityCondition}), $V_\circ$ will have no edge connecting to $\alpha$ in $V_\bullet$.   Else, $\Lambda_K(\alpha)=(\circ,\bullet)$, so there must exist $\alpha_1\in \gamma$ which shares a left endpoint with $\alpha$ and a right endpoint $>k$.  But, to avoid contradiction $\Lambda_K(\alpha_1)=(\bullet,\circ)$, so there must exist $\alpha_2\in \gamma$ sharing the right endpoint with $\alpha_1$ and with left endpoint left of $\alpha$. When we iterate, we obtain a picture of the form
  $$
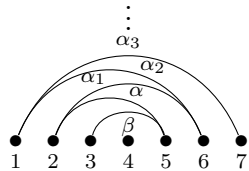
\begin{tikzpicture}[baseline=.3cm]
	\foreach \x in {1,...,7} 
		\node (\x) at (\x/2,0) [inner sep=0pt] {$\bullet$};
	\foreach \x in {1,...,7} 
		\node at (\x/2,-.25) {$\scs \x$};
	\node at (2,1.75) {$\scs\vdots$};
	\draw (1) .. controls (2.25/2,1.5) and (5.75/2,1.5) .. node [above=-1pt] {$\scs\alpha_3$}  (7); 
	\draw (1) .. controls (2/2,1.25) and (5/2,1.25) .. node [above=-1pt,pos=.7] {$\scs\alpha_2$}  (6); 
	\draw (2) .. controls (2.75/2,1) and (5.25/2,1) .. node [above=-1pt,pos=.3] {$\scs\alpha_1$}  (6); 
	\draw (2) .. controls (2.5/2,.75) and (4.5/2,.75) .. node [above=-1pt,pos=.7] {$\scs\alpha$}  (5); 
	\draw (3) .. controls (3.25/2,.5) and (4.75/2,.5) .. node [below=-2pt] {$\scs\beta$}  (5); 
\end{tikzpicture}$$
Since there are a finite number of nodes,  eventually there must be $\alpha_f$ with $\Lambda_K(\alpha_f)=(\bullet,\bullet)$, which leads to the same contradiction as before.  Thus, if $\langle \chi^\gamma, \One\rangle\neq 0$, then $\Lambda_K(\gamma)\subseteq \{(\circ,\bullet),(\bullet,\circ)\}$. 

Checking the converse is straightforward.
\end{proof}

By Lemma \ref{NonzeroOneLemma}, if $\gamma\in\cI_K(\lambda)$ and $\langle \chi^\gamma,\One\rangle\neq 0$, then $\Lambda_K(\gamma)\subseteq \{(\bullet,\circ),(\circ,\bullet)\}$.  It follows that if $\tilde\gamma\in\cS_{K'\cup L'}(q)$ is obtained by applying (SL), (SR), and (SB) to $\gamma$ as in Theorem \ref{StraighteningTheorem}, then $j\larc{a}k\in \tilde\gamma$ only if $|\{j,k\}\cap K'|=1$.  Thus, we can iterate by finding $\gamma_2\in\cI_{K'}(\tilde\gamma)$.  
This process gives a sequence $(\gamma_0,\gamma_1,\ldots,\gamma_\ell)$ with $\gamma_0=\lambda$ which terminates at $\gamma_\ell=\emptyset$.  In fact,
\begin{equation}\label{SequenceIdentity}
\langle \Res_{U_K}^{U_L}(\chi^\lambda),\One\rangle=q^{r_K^L(\lambda)}\left|\left\{(\gamma_0,\ldots, \gamma_\ell)\bigg| \begin{array}{@{}l@{}} \gamma_0=\lambda, \gamma_\ell=\emptyset,\ell\in\ZZ_{\geq 0}, \\ \gamma_{k+1}\in \cI_{K_k}(\tilde\gamma_k), K_0=K\end{array}\right\}\right|,
\end{equation}
where $\tilde\gamma_j \in \cS_{K_j\cup L_j}(q)$ is a set-partition obtained by applying (SL), (SR), (SB) to $\gamma_j$ such that 
$$\chi^{\gamma_j}=\Res_{U_{K_j}}^{U_{K_j\cup L_j}}(\chi^{\tilde\gamma_j}).$$

\begin{remark}
Note that the sequence $(\gamma_0,\ldots, \gamma_\ell)$ is not a path in $\cP$.  Even if we reindex the endpoints of $\gamma_j$ via the order preserving  bijection $K_{j-1}\rightarrow K$, the resulting sequence will not be a path in $\cP$.  However, after re-indexing the sequence will be a sub-path of a path in $\cP$ (skipping some steps).  That is, we avoid the difficulties of choosing a canonical order for resolving conflicts by resolving the conflicts in ``clumps."  
\end{remark}

The final two lemmas more closely study the relationship between $\gamma$ and $\lambda$, where $\gamma\in\cI_K(\lambda)$.

\begin{lemma}  \label{GoingToHalfLemma}
Let $K\subseteq L\subseteq \ZZ_{\geq 1}$ be finite sets.  Suppose $\lambda\in\cS_L(q)$ satisfies $j\larc{a}k\in\lambda$ only if $|\{j,k\}\cap K|=1$, and let $\gamma\in\cI_K(\lambda)$ with $\Lambda_K(\gamma)\subseteq \{(\circ,\bullet),(\bullet,\circ)\}$.  Then for each $j\larc{a}k\in\gamma$ with $m_{jk}(\gamma)=1$
\begin{enumerate}
\item[(a)]  $\Lambda_K(j\larc{a}k)=(\circ,\bullet)$ if and only if $j\larc{a}k\in \cI_{i\slarc{a'}k}(\lambda)\cap \gamma$ for some $i<j$,
\item[(b)]  $\Lambda_K(j\larc{a}k)=(\bullet,\circ)$ if and only if  $j\larc{a}k\in\cI_{j\slarc{a'}l}(\lambda)\cap \gamma$ for some $l>k$.
\end{enumerate}
\end{lemma}

\begin{proof}
We first prove
\begin{enumerate}
\item[(a')] If $\Lambda_K(j\larc{a}k)=(\circ,\bullet)$, then $j\larc{a}k\in \cI_{i\slarc{a'}k}(\lambda)\cap \gamma$ for some $i<j$,
\item[(b')] If $\Lambda_K(j\larc{a}k)=(\bullet,\circ)$, then $j\larc{a}k\in\cI_{j\slarc{a'}l}(\lambda)\cap \gamma$ for some $l>k$.
\end{enumerate}
Suppose  $k-j$ is  maximal such that $j\larc{a}k\in \gamma$ is a counterexample to (a') or (b').  WLOG assume that $\Lambda_K(j\larc{a}k)=(\circ,\bullet)$.  Since $j\larc{a}k$ is a counterexample, there exists $l>k$ such that $j\larc{a}k\in \cI_{j\slarc{a'}l}(\lambda)$.    However, $\Lambda_K(j\larc{a}k)=(\circ,\bullet)$ implies there exists $j\larc{b}l'\in \gamma$ such that $l'>k$.  Let $l'$ be maximal with this property.  Since $j\larc{a'}l\in \lambda$, there must be $i<j$ such that $j\larc{b}l'\in \cI_{i\slarc{b'}l'}(\lambda)$.  However, $\Lambda_K(j\larc{b}l')\in \{(\circ,\bullet),(\bullet,\circ)\}$, and by the maximality of $l'$, $\Lambda_K(j\larc{b}l')=(\bullet,\circ)$, which contradicts (b).  Since  $l'-j>k-j$ this contradicts the maximality of our counterexample, proving (a') and (b').

Conversely, suppose $j\larc{a}k\in \cI_{i\slarc{a'}k}(\lambda)$ for some $i<j$.   If $\Lambda_K(j\larc{a}k)=(\bullet, \circ)$, then by (b') there exists $l>k$ such that $j\larc{a}k\in \cI_{j\slarc{b}l}(\lambda)$.  However, in this case $m_{jk}(\lambda)=2$.  Thus, $\Lambda_K(j\larc{a}k)=(\circ,\bullet)$. The converse of (b') is proved similarly.
\end{proof}

The last lemma will help us keep track of elements of $\cC_K(\lambda)$.

\begin{lemma} \label{ComeFromCrossingLemma}
Let $K\subseteq L\subseteq \ZZ_{\geq 1}$ be finite sets.  Suppose $\lambda\in\cS_L(q)$ satisfies $j\larc{a}k\in\lambda$ only if $|\{j,k\}\cap K|=1$, and let $\gamma\in\cI_K(\lambda)$ with $\Lambda_K(\gamma)\subseteq \{(\circ,\bullet),(\bullet,\circ)\}$.  If $i\larc{a}k,i\larc{b}j\in \gamma$ or $i\larc{a}k,j\larc{b}k\in \gamma$ with $i<j<k$, then there exist $h<i$ and $l>k$ such that $(h\larc{a'}k,i\larc{b'}l)\in \cC_K(\lambda)$.
\end{lemma}

\begin{proof}
WLOG, assume  $i\larc{a}k,i\larc{b}j\in \gamma$ with $i<j<k$.   By Lemma \ref{GoingToHalfLemma}, $\gamma\in \cI_K(\lambda)$ implies 
$$(i\larc{a}k,i\larc{b}j)\in \cI_{i\slarc{a'}l}(\lambda)\times\cI_{h\slarc{b'}j}(\lambda)\quad   \text{or}\quad  (i\larc{a}k,i\larc{b}j)\in \cI_{h'\slarc{a'}k}(\lambda)\times\cI_{h\slarc{b'}j}(\lambda).$$
In the first case, we are done.

Suppose $k$ is maximal such that  $ (i\larc{a}k,i\larc{b}j)\in \cI_{h'\slarc{a'}k}(\lambda)\times\cI_{h\slarc{b'}j}(\lambda)$.   If there is no $i\larc{b'}l\in \lambda$ such that $l>k$ and $l\notin K$, then $m_{ik}(\gamma)=1$ so  by Lemma \ref{GoingToHalfLemma}(a) we have $\Lambda_K(i\larc{a}k)=(\circ,\bullet)$.  Thus, there must exist $h''<i$ and $l>k$ such that $i\larc{c}l\in \gamma\cap \cI_{h''\slarc{c'}l}(\lambda)$.  However, this contradicts the maximality of $k$.
\end{proof}

The proof of Theorem \ref{ExplicitCoefficientsTheorem} uses observation (\ref{SequenceIdentity}), and proves that 
$$\cR_K(\lambda)=\left\{(\gamma_0,\ldots, \gamma_\ell)\bigg| \begin{array}{@{}l@{}} \gamma_0=\lambda, \gamma_\ell=\emptyset,  \ell\in\ZZ_{\geq 0},\\ \gamma_{k+1}\in \cI_{K_k}(\tilde\gamma_k), K_0=K\end{array}\right\}\longleftrightarrow \{f:\cC_K(q)\rightarrow \FF_q\}.$$

\begin{proof}[Proof of Theorem \ref{ExplicitCoefficientsTheorem}]
Suppose $\lambda\in\cS_L(q)$ satisfies $j\larc{a}k\in\lambda$ only if $|\{j,k\}\cap K|=1$.
  
Let $(\gamma_0,\gamma_1,\gamma_2,\ldots, \gamma_\ell)\in \cR_K(\lambda)$.  We wish to construct $f:\cC_K(q)\rightarrow \FF_q$ using this sequence.  First, note that for $(i\larc{a}k,j\larc{b}l)\in \cC(\gamma_m)$ with $\Lambda_{K_m}(i\larc{a}k)=(\circ,\bullet)$ and $\Lambda_{K_m}(j\larc{b}l)=(\bullet,\circ)$, by Lemma \ref{GoingToHalfLemma} there exists $i'<i$ and $l'>l$ such that $(i'\larc{a'}k,j\larc{b'}l')\in \cC_{K_m}(\tilde\gamma_{m-1})$.  Furthermore, if  $\tau:K_{m+1}\rightarrow K_m$ is the order preserving bijection, then  $(i\larc{a}k,j\larc{b}l)\in \cC_{K_{m+1}}(\tilde\gamma_m)$ if and only if there exists $i'<\tau(j)$ and $l'>\tau(k)$ such that $(i'\larc{a}\tau(k),\tau(j)\larc{b}l')\in \cC(\gamma_m)$ with $\Lambda_{K_m}(i'\larc{a}\tau(k))=(\circ,\bullet)$ and $\Lambda_{K_m}(\tau(j)\larc{b}l')=(\bullet,\circ)$.  By composing these identifications we get injective maps 
$$\iota_m:\cC_{K_m}(\tilde\gamma_m)\rightarrow \cC_{K}(\lambda).$$

The following rules explain how to define $f:\cC_K(\lambda)\setminus \iota_1(\cC_{K_1}(\gamma_1))\rightarrow \FF_q$.  Since each $\gamma_m\in \cI_{K_{m-1}}(\tilde\gamma_{m-1})$, we can iterate to find the remaining values of $f$. 
\begin{enumerate}
\item[(1)] If $m_{jk}(\gamma)=2$, then the pair $(j\larc{a}k,j\larc{-a}k)\in \cI_{i\slarc{b}k}(\lambda)\times\cI_{j\slarc{c}l}(\lambda)$ for some $i<j<k<l$ with $i,l\notin K$.  Define $f(i\larc{b}k,j\larc{c}l)=a$.  
\item[(2)] If $m_{jk}(\gamma)=1$, then by Lemma \ref{HalfInLemma} and Lemma \ref{ComeFromCrossingLemma} either 
\begin{itemize}
\item $\{j\larc{c}l,j\larc{d}k\}\subseteq \gamma$ with $k<l$ and there exists  $(j\larc{a}m,i\larc{b}k)\in \cC_K(\lambda)$, 
\item  $\{i\larc{c}k,j\larc{d}k\}\subseteq \gamma$ with $i<j$ and there exists $(h\larc{a}k,j\larc{b}l)\in \cC_K(\lambda)$.  
\end{itemize}
In the first case, $f(j\larc{a}m,i\larc{b}k)=d$, and in the second case, $f(h\larc{a}k,j\larc{b}l)=d$.
\item[(3)] Suppose $m_{jk}(\gamma)=0$, and either 
\begin{itemize}
\item The pair $(j\larc{c}l,\emptyset)\in \cI_{j\slarc{a}m}(\lambda)\times\cI_{i\slarc{b}k}(\lambda)$ with $(j\larc{a}m,i\larc{b}k)\in \cC_K(\lambda)$, 
\item The pair $(i\larc{c}k,\emptyset)\in \cI_{h\slarc{a}k}(\lambda)\times\cI_{j\slarc{b}l}(\lambda)$ with $(h\larc{a}k,j\larc{b}l)\in \cC_K(\lambda)$.
\end{itemize}
Then $f(j\larc{a}m,i\larc{b}k)=0$ or $f(h\larc{a}k,j\larc{b}l)=0$, respectively.
\end{enumerate}
For example, suppose $K=\{4,5,7,8,9,10\}$, and we have 
$$(\gamma_0, \gamma_1, \gamma_2,\gamma_3)=\bigg(
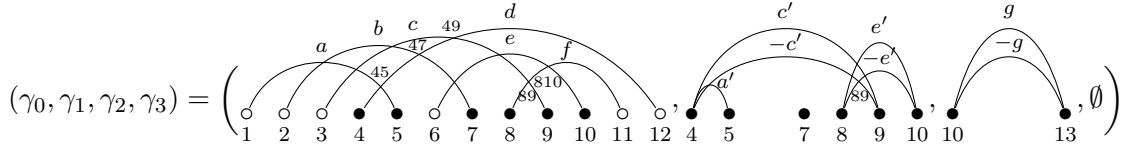
\begin{tikzpicture}[baseline = .1cm]
    \foreach \y in {1,2,3,6,11,12}
        \node(\y) at (\y/2,0) [inner sep=0pt] {$\circ$};
    \foreach \z in {4,5,7,8,9,10}
        \node(\z) at (\z/2,0) [inner sep=0pt] {$\bullet$};
    \foreach \x in {1,...,12}
        \node at (\x/2,-.25) [inner sep=0pt] {$\scs\x$};
    \draw (1) .. controls (1.75/2, .6*1.5) and (4.25/2, .6*1.5) .. node [above] {$\scs a $} (5);
    \draw (2) .. controls (3/2, .8*1.5) and (6/2, .8*1.5) .. node [above] {$\scs b $}  (7);
    \draw (3) .. controls (4.5/2, .9*1.5) and (7.75/2, .9*1.5) .. node [above,pos=.4] {$\scs c $}  (9);
    \draw (4) .. controls (6/2, 1*1.5) and (10/2, 1*1.5) .. node [above] {$\scs d $}  node [above=-2pt,pos=.21] {$\scscs 47$} node [above,pos=.07] {$\scscs 45$} node [above=-1pt,pos=.32] {$\scscs 49$} (12);
    \draw (6) .. controls (7/2, .7*1.5) and (9/2, .7*1.5) .. node [above] {$\scs e $}  (10) ;
    \draw (8) .. controls (8.75/2, .6*1.5) and (10.25/2, .6*1.5) .. node [above=-3pt] {$\scs f $}  node [below=-1pt,pos=.16] {$\scscs 89$} node [below,pos=.35] {$\scscs 810$} (11);
\end{tikzpicture}, 
\begin{tikzpicture}[baseline = .1cm]
    \foreach \z in {4,5,7,8,9,10}
        \node(\z) at (\z/2,0) [inner sep=0pt] {$\bullet$};
    \foreach \x in {4,5,7,8,9,10}
        \node at (\x/2,-.25) [inner sep=0pt] {$\scs\x$};
    \draw (4) .. controls (4.25/2, .5) and (4.75/2, .5) .. node [above=2pt,pos=.95] {$\scs a' $} (5);
    \draw (4) .. controls (4.75/2, 1) and (8.25/2, 1) .. node [above=-2pt] {$\scs -c' $}  (9);
    \draw (4) .. controls (4.75/2, 1.5) and (8.25/2, 1.5) .. node [above] {$\scs c' $}   (9);
    \draw (8) .. controls (8.5/2, .75) and (9.5/2, .75) .. node [above=-2pt] {$\scs -e' $} node [below,pos=.25] {$\scscs 89$} (10) ;
    \draw (8) .. controls (8.5/2, 1.25) and (9.5/2, 1.25) .. node [above] {$\scs e' $}   (10);
\end{tikzpicture},\begin{tikzpicture}[baseline = .1cm]
    \foreach \y in {10,13}
        \node(\y) at (\y/2,0) [inner sep=0pt] {$\bullet$};
    \foreach \x in {10,13}
        \node at (\x/2,-.25) [inner sep=0pt] {$\scs\x$};
    \draw (10) .. controls (10.75/2, 1.5) and (12.25/2, 1.5) .. node [above] {$\scs g$}   (13);
    \draw (10) .. controls (10.75/2, 1) and (12.25/2, 1) .. node [above=-2pt] {$\scs -g$}   (13) ;
\end{tikzpicture},\emptyset\bigg) $$
where the elements of $\cC_K(q)$ are labeled by pairs $jk$.  Apply (1--3) to $\lambda$ and $\gamma_1$ to obtain $f(49)=c'$, $f(810)=e'$, $f(45)=a'$, $f(47)=0$, and $f(89)$ is undefined. If we now apply (1--3) to
$$\tilde\gamma_1=
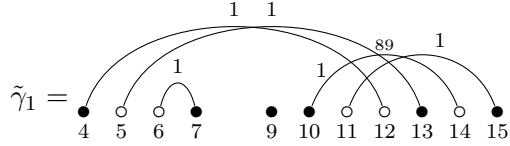
\begin{tikzpicture}[baseline = .1cm]
    \foreach \y in {4,7,9,10,13,15}
        \node(\y) at (\y/2,0) [inner sep=0pt] {$\bullet$};
    \foreach \z in {5,6,11,12,14}
        \node(\z) at (\z/2,0) [inner sep=0pt] {$\circ$};
    \foreach \x in {4,...,7}
        \node at (\x/2,-.25) [inner sep=0pt] {$\scs\x$};
    \foreach \x in {9,...,15}
        \node at (\x/2,-.25) [inner sep=0pt] {$\scs\x$};
    \draw (6) .. controls (6.25/2, .5) and (6.75/2, .5) .. node [above] {$\scs 1$} (7);
    \draw (4) .. controls (5/2, 1.5) and (11/2, 1.5) .. node [above] {$\scs 1$}  (12);
    \draw (5) .. controls (6/2, 1.5) and (12/2, 1.5) .. node [above] {$\scs 1$}   (13);
    \draw (10) .. controls (10.75/2, 1) and (13.25/2, 1) .. node [above,pos=.1] {$\scs 1$} node [above=-2pt] {$\scscs 89$}  (14) ;
    \draw (11) .. controls (11.75/2, 1) and (14.25/2, 1) .. node [above,pos=.6] {$\scs 1$}   (15);
\end{tikzpicture}$$ 
and $\gamma_2$, we obtain that $f(89)=g$.  Another iteration gives no new information, there are no more significant crossings in 
$$\tilde\gamma_2=
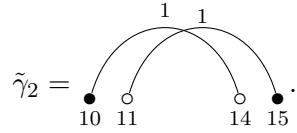
\begin{tikzpicture}[baseline = .1cm]
    \foreach \y in {10,15}
        \node(\y) at (\y/2,0) [inner sep=0pt] {$\bullet$};
    \foreach \y in {11,14}
        \node(\y) at (\y/2,0) [inner sep=0pt] {$\circ$};
    \foreach \x in {10,11,14,15}
        \node at (\x/2,-.25) [inner sep=0pt] {$\scs\x$};
    \draw (10) .. controls (10.75/2, 1.25) and (13.25/2, 1.25) .. node [above] {$\scs 1$}   (14);
    \draw (11) .. controls (11.75/2, 1.25) and (14.25/2, 1.25) .. node [above=-2pt] {$\scs 1$}   (15) ;
\end{tikzpicture}.$$

Conversely, suppose $f:\cC_K(\lambda)\rightarrow \FF_q$ is a function.  Then $f$ recursively determines a sequence $(\gamma_0,\gamma_1,\ldots, \gamma_\ell)$ as follows.  Let $\gamma_0=\lambda$.  Then for $1\leq m\leq \ell$, 
\begin{description}
\item[Step 0.] Let $\tilde\gamma_{m-1}\in \cS(q)$ be obtained from $\gamma_{m-1}$ via straightening rules (SL), (SR), (SB).  Let $f_\iota=f\circ\iota_{m-1}$, $B=\{c\in\cC_K(\tilde\gamma_{m-1})\mid f_{\iota}(c)\neq 0\}$, $A=\emptyset$, and $\gamma_m=\emptyset$.
\item[Step 1.]  For each maximal crossing $c = (i\larc{a}k, j\larc{b} l)\in B$, set
$$\gamma_m=\gamma_m\cup \{j\larc{f_\iota(c)}k,j\larc{-f_\iota(c)}k\},$$
$B=B-\{c\}$, and $A=A\cup\{c\}$.
\item[Step 2.]If $B=\emptyset$, then we are done.  Else, for each maximal crossing $c = (i\larc{a} k, j\larc{b} l)\in B$,  if either $i\larc{a} k$ or $j\larc{b}l$ is involved in no crossing in $A$, then 
$$\gamma_m=\gamma_m\cup\{j\larc{f_\iota(c)} k\},\quad  B=B-\{c\},\quad \text{and}\quad A=A\cup\{c\}.$$
Otherwise, $B=B-\{c\}$.
\item[Step 3.] Repeat Step 2.
\end{description}
For example, suppose
$$\lambda=
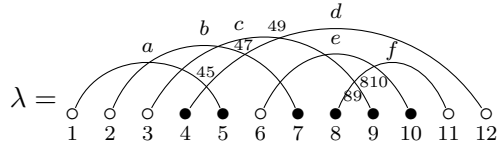
\begin{tikzpicture}[baseline = .1cm]
    \foreach \y in {1,2,3,6,11,12}
        \node(\y) at (\y/2,0) [inner sep=0pt] {$\circ$};
    \foreach \z in {4,5,7,8,9,10}
        \node(\z) at (\z/2,0) [inner sep=0pt] {$\bullet$};
    \foreach \x in {1,...,12}
        \node at (\x/2,-.25) [inner sep=0pt] {$\scs\x$};
    \draw (1) .. controls (1.75/2, .6*1.5) and (4.25/2, .6*1.5) .. node [above] {$\scs a $} (5);
    \draw (2) .. controls (3/2, .8*1.5) and (6/2, .8*1.5) .. node [above] {$\scs b $}  (7);
    \draw (3) .. controls (4.5/2, .9*1.5) and (7.75/2, .9*1.5) .. node [above,pos=.4] {$\scs c $}  (9);
    \draw (4) .. controls (6/2, 1*1.5) and (10/2, 1*1.5) .. node [above] {$\scs d $}  node [above=-2pt,pos=.21] {$\scscs 47$} node [above,pos=.07] {$\scscs 45$} node [above=-1pt,pos=.32] {$\scscs 49$} (12);
    \draw (6) .. controls (7/2, .7*1.5) and (9/2, .7*1.5) .. node [above] {$\scs e $}  (10) ;
    \draw (8) .. controls (8.75/2, .6*1.5) and (10.25/2, .6*1.5) .. node [above=-3pt] {$\scs f $}  node [below=-1pt,pos=.16] {$\scscs 89$} node [below,pos=.35] {$\scscs 810$} (11);
\end{tikzpicture}$$
and $f(45)=3$, $f(47)=1$, $f(49)=f(810)=0$ and $f(89)=5$.   To compute $\gamma_1$, we have $B=\{45,47,89\}$.   Step 1 gives
$$\gamma_1=\{4\larc{1}7,4\larc{-1}7,8\larc{5}9,8\larc{-5}9\},\quad B=\{45\}, \quad A=\{47,89\}.$$
Then for Step 2, $45=(1\larc{a}5,4\larc{d}12)$ satisfies $1\larc{a} 5$ is not involved in any crossing in $A$, so
$$\gamma_1=\{4\larc{1}7,4\larc{-1}7,8\larc{5}9,8\larc{-5}9,4\larc{3}5\}, \quad B=\{\}, \quad A=\{44,47,89\}.$$
Since
$$\gamma_1=\begin{tikzpicture}[baseline = .1cm]
    \foreach \z in {4,5,7,8,9,10}
        \node(\z) at (\z/2,0) [inner sep=0pt] {$\bullet$};
    \foreach \x in {4,5,7,8,9,10}
        \node at (\x/2,-.25) [inner sep=0pt] {$\scs\x$};
    \draw (4) .. controls (4.25/2, .5) and (4.75/2, .5) .. node [above=-1pt,pos=.8] {$\scs 3 $} (5);
    \draw (4) .. controls (4.75/2, 1) and (6.25/2, 1) .. node [above=-2pt] {$\scs -1 $}  (7);
    \draw (4) .. controls (4.75/2, 1.5) and (6.25/2, 1.5) .. node [above] {$\scs 1 $}  (7);
    \draw (8) .. controls (8.25/2, .5) and (8.75/2, .5) .. node [above=-2pt,pos=.4] {$\scs -5 $}  (9) ;
    \draw (8) .. controls (8.25/2, 1) and (8.75/2, 1) .. node [above] {$\scs 5 $}  (9) ;
\end{tikzpicture},
\qquad\text{we have}\qquad \tilde\gamma_1=
\begin{tikzpicture}[baseline = .1cm]
    \foreach \z in {4,7,10,11,14,15}
        \node(\z) at (\z/2,0) [inner sep=0pt] {$\bullet$};
    \foreach \z in {5,6,9,12,13}
        \node(\z) at (\z/2,0) [inner sep=0pt] {$\circ$};
    \foreach \x in {4,...,7}
        \node at (\x/2,-.25) [inner sep=0pt] {$\scs\x$};
    \foreach \x in {9,...,15}
        \node at (\x/2,-.25) [inner sep=0pt] {$\scs\x$};
    \draw (6) .. controls (6.25/2, .5) and (6.75/2, .5) .. node [above=-1pt] {$\scs 1 $} (7);
    \draw (4) .. controls (4.75/2, 1.25) and (8.25/2, 1.25) .. node [above=-2pt] {$\scs 1 $}  (9);
    \draw (5) .. controls (5.75/2, 1.25) and (9.25/2, 1.25) .. node [above] {$\scs 1 $}  (10);
    \draw (11) .. controls (11.5/2, .5) and (12.5/2, .5) .. node [above=-2pt,pos=.4] {$\scs 1 $}  (13) ;
    \draw (12) .. controls (12.5/2, .5) and (13.5/2, .5) .. node [above] {$\scs 1 $}  (14) ;
\end{tikzpicture},$$
and the by Lemma \ref{HalfInLemma}, $\gamma_2=\emptyset$.

Since these two algorithms invert one-another, by (\ref{SequenceIdentity})
$$\langle \Res_{U_K}^{U_L}(\chi^\lambda),\One\rangle = q^{r_K^L(\lambda)}|\{f:\cC_K(\lambda)\rightarrow \FF_q\}|=q^{r_K^L(\lambda)}q^{|\cC_K(\lambda)|},$$
as desired.
\end{proof}

We obtain an analogous result for multisets $\lambda\in \cM_K(q)$ (in particular, for tensor products), where the correct generalization of $\cC_K(\lambda)$ is 
$$\cC_K(\lambda)=\{(i\larc{a}k,j\larc{b}l)\in \cC(\lambda)\mid i<j<k<l, \Lambda_K(i\larc{a}k)=(\circ,\bullet),\Lambda_K(j\larc{b}l)=(\bullet,\circ)\}.$$

\begin{corollary}
Let $K\subseteq \ZZ_{\geq 1}$ be  finite subset and $\lambda\in \cM_K(q)$ with $\Lambda_K(\lambda)\in \{(\circ,\bullet),(\bullet,\circ)\}$.  Then
$$\langle \chi^\lambda,\One\rangle= q^{|\cC_K(\lambda)|}.$$
\end{corollary}

\begin{proof} 
By Theorem \ref{StraighteningTheorem} there exists $\tilde\lambda\in \cS_{K'\cup L'}(q)$ such that 
$$\chi^{\lambda}=q^{-r_{K'}^{K'\cup L'}(\tilde\lambda)}\Res_{U_{K'}}^{U_{K'\cup L'}}(\chi^{\tilde\lambda}).$$
Apply Theorem \ref{ExplicitCoefficientsTheorem} to get the desired result.
\end{proof}

\begin{remark}
In general, the coefficient of $\One$ will not be a power of $q$, but merely polynomial in $q$, as demonstrated by
$$\langle \chi^{\{1\slarc{a}5,1\slarc{b} 5,1\slarc{c}5\}},\One\rangle_{U_5(q)}=3q-2, \quad\text{when $a+b+c=0$.}$$
\end{remark}

\end{document}